\documentclass[12pt,a4paper]{amsart}
\usepackage[UKenglish]{babel}
\usepackage[utf8]{inputenx}
\usepackage[mathscr]{eucal}
\usepackage{indentfirst}
\usepackage{newlfont}
\usepackage{verbatim}
\usepackage{lmodern,textcomp}
\usepackage{fp}
\usepackage{booktabs}
\usepackage{ragged2e}
\usepackage{longtable}
\usepackage{epic}
\usepackage[margin=2.9cm]{geometry}
\usepackage{epstopdf} 
\usepackage{multicol}
\usepackage{graphicx}
\usepackage{pict2e}
\usepackage{caption}
\usepackage{fancyhdr}
\usepackage{wrapfig}
\usepackage{subfigure}
\usepackage{xcolor}
\usepackage{colortbl}
\usepackage[most]{tcolorbox}
\usepackage{epstopdf}
\usepackage{empheq} 
\usepackage{amsfonts}
\usepackage{amsthm}
\numberwithin{equation}{section}
\usepackage{amsmath}
\usepackage{amscd}
\usepackage{amssymb}
\usepackage{mathtools}
\usepackage{leftidx}
\usepackage{mathabx}
\usepackage{accents}
\usepackage[shortlabels]{enumitem}
\usepackage[makeroom]{cancel}
\usepackage{multirow,bigdelim}
\usepackage{array}
\usepackage{hhline}
\usepackage{mathptmx}
\usepackage{calc}
\usepackage{hyperref} 
\usepackage{tikz-cd}
\usepackage{tikz}
\usetikzlibrary{positioning,tikzmark,calc,,arrows,shapes,decorations.pathreplacing}
\tikzset{every picture/.style={remember picture}}
\usetikzlibrary{shadows,arrows}
\pgfdeclarelayer{background}
\pgfdeclarelayer{foreground}
\pgfsetlayers{background,main,foreground}
\tikzstyle{materia}=[draw, fill=blue!20, text width=6.0em, text centered,  minimum height=1.5em,drop shadow]
\tikzstyle{practica} = [materia, text width=8em, minimum width=10em,  minimum height=3em, rounded corners, drop shadow]
\tikzstyle{appe} = [materia, text width=8em, minimum width=10em,  minimum height=3em, rounded corners, drop shadow]
\tikzstyle{texto} = [above, text width=6em, text centered]
\tikzstyle{linepart} = [draw, thick, color=black!50, -latex', dashed]
\tikzstyle{line} = [draw, thick, color=black!50, -latex']
\tikzstyle{ur}=[draw, text centered, minimum height=0.01em]
\definecolor{orcidlogocol}{HTML}{A6CE39}
\makeatletter
\def\@tocline#1#2#3#4#5#6#7{\relax
  \ifnum #1>\c@tocdepth 
  \else
    \par \addpenalty\@secpenalty\addvspace{#2}%
    \begingroup \hyphenpenalty\@M
    \@ifempty{#4}{%
      \@tempdima\csname r@tocindent\number#1\endcsname\relax
    }{%
      \@tempdima#4\relax
    }%
    \parindent\z@ \leftskip#3\relax \advance\leftskip\@tempdima\relax
    \rightskip\@pnumwidth plus4em \parfillskip-\@pnumwidth
    #5\leavevmode\hskip-\@tempdima
      \ifcase #1
       \or\or \hskip 1em \or \hskip 2em \else \hskip 3em \fi%
      #6\nobreak\relax
    \dotfill\hbox to\@pnumwidth{\@tocpagenum{#7}}\par
    \nobreak
    \endgroup
  \fi}
\makeatother
\newcommand{\mycomment}[1]{%
}

\newcommand*{\norm}[1]{\left\lVert#1\right\rVert}
\newcommand*{\vertiii}[1]{{\left\vert\kern-0.25ex\left\vert\kern-0.25ex\left\vert #1 \right\vert\kern-0.25ex\right\vert\kern-0.25ex\right\vert}}

\newcommand*{\mres}{\mathbin{\vrule height 1.6ex depth 0pt width 0.13ex\vrule height 0.13ex depth 0pt width 1.3ex}}
\makeatletter
\newcommand\mathcircled[1]{%
  \mathpalette\@mathcircled{#1}%
}
\newcommand\@mathcircled[2]{%
  \tikz[baseline=(math.base)] \node[draw,circle,inner sep=1pt] (math) {$\m@th#1#2$};%
}
\makeatother
 
\allowdisplaybreaks 
\DeclareMathAlphabet{\mathcal}{OMS}{cmsy}{m}{n}
\DeclareMathOperator{\spt}{spt}

\DeclareMathOperator{\spn}{span}
\DeclareMathOperator{\sgn}{sgn}
\theoremstyle{plain}
\newtheorem{theor}{Theorem}[section] 
\newtheorem{lem}[theor]{Lemma} 
\newtheorem{cor}[theor]{Corollary}
\newtheorem{prop}[theor]{Proposition}

\newtheorem*{prop_nonumber}{Proposition}
\newtheorem*{cor*}{Corollary}
\newtheorem*{prop*}{Proposition}
\theoremstyle{definition}
\newtheorem{defin}[theor]{Definition}
\newtheorem*{defin_nonumber}{Definition}
\newtheorem*{defin*}{Definition}
\newtheorem{rem}[theor]{Remark}
\newtheorem{ex}[theor]{Example}
\newtheorem{obs}[theor]{Observation}
\newtheorem{no}[theor]{Notation}

\begin{document}
\title{Sub-Riemannian Currents and Slicing of Currents in the Heisenberg group $\mathbb{H}^n$}
\author[G. Canarecci]{Giovanni Canarecci}
\address{University of Helsinki \\ Department of Mathematics and Statistics \\
Helsinki, Finland} 
\email{giovanni.canarecci@helsinki.fi}
\begin{abstract} 
This paper aims to define and study currents and slices of currents in the Heisenberg group $\mathbb{H}^n$.  
Currents, depending on their integration properties and on those of their boundaries, can be classified into subspaces and, assuming their support to be compact, we can work with currents of finite mass, define the notion of slices of Heisenberg currents and show some important properties for them.   
While some such properties are similarly true in Riemannian settings, others carry deep consequences because they do not include the slices of the middle dimension $n$, which opens new challenges and scenarios for the possibility of developing a compactness theorem.  
Furthermore, this suggests that the study of currents on the first Heisenberg group $\mathbb{H}^1$ diverges from the other cases, because that is the only situation in which the dimension of the slice of a hypersurface, $2n-1$, coincides with the middle dimension $n$, which triggers a change in the associated differential operator in the Rumin complex.
\end{abstract}
\maketitle
\tableofcontents


\section*{Introduction} 
The aim of this paper is to define and study currents and slices of currents in the Heisenberg group $\mathbb{H}^n$ to provide tools for developing a compactness theorem for such currents.\\
There exist many references for an introduction to the Heisenberg group; here we follow mainly sections 2.1 and 2.2 in \cite{FSSC} and sections 2.1.3 and 2.2 in \cite{CDPT}.   
The Heisenberg group $\mathbb{H}^n$, $n \geq 1$, is the $(2n+1)$-dimensional manifold $ \mathbb{R}^{2n+1}$ with a non-Abelian group product and the Carnot--Carath{\'e}odory distance (or the equivalent  Kor{\'a}nyi distance).   
Additionally, the Heisenberg group is a Carnot group of step $2$ with Lie algebra $\mathfrak{h} = \mathfrak{h}_1 \oplus \mathfrak{h}_2$. 
The \textit{horizontal} layer $\mathfrak{h}_1$ has a standard orthonormal basis of left invariant vector fields,  
$ X_j=\partial_{x_j} -\frac{1}{2} y_j \partial_t$ and $Y_j=\partial_{y_j} +\frac{1}{2} x_j \partial_t$ for $ j=1,\dots,n$,     
which hold the core property that $[X_j, Y_j] = \partial_t=:T $ for each $j$. 
$T$ alone spans the second layer $\mathfrak{h}_2$ and is called the \textit{vertical} direction.  
The Heisenberg group has a natural cohomology called Rumin cohomology (see Rumin \cite{RUMIN}), whose behaviour is significantly different from the standard de Rham one (see also \cite{GClicentiate}). In the Rumin cohomology, the complex is given not by one but by three operators, depending on the dimension:
\begin{defin_nonumber}[\ref{complexHn}]    
Given the definitions in Section \ref{sec:preliminaries}, the Rumin complex is given by
$$
0 \to \mathbb{R} \to C^\infty  \stackrel{d_Q}{\to} \frac{\Omega^1}{I^1}  \stackrel{d_Q}{\to}  \dots \stackrel{d_Q}{\to} \frac{\Omega^n}{I^n} \stackrel{D}{\to} J^{n+1}    \stackrel{d_Q}{\to} \dots   \stackrel{d_Q}{\to} J^{2n+1} \to 0,
$$
where $d$ is the standard differential operator and, for $k < n$,
$$
d_Q( [\alpha]_{I^k} ) :=  [d \alpha]_{I^{k+1}},
$$
while, for $k \geq n +1$,
$$
d_Q := d_{\vert_{J^k}}.
$$
The second order differential operator $D$ is defined as
$$
D( [\alpha]_{I^n} ) :=  d \left ( \alpha +  L^{-1} \left (- (d \alpha)_{\vert_{ {\prescript{}{}\bigwedge}^{n+1} \mathfrak{h}_1 }} \right ) \wedge \theta \right )
=  d \left ( \alpha +  \mathcal{L} (\alpha)  \wedge \theta \right ).
$$
\end{defin_nonumber}
In Section \ref{currents} we define the notion of current in the Heisenberg group and show how one can think them, only to fix the idea, as special Riemannian currents.   
Then we describe how a current $T$ can be written as integral with the notion of \emph{representability by integration}, denoted $T=\overrightarrow{T} \wedge \mu_T$, we define its mass $M(T)$ and show that finite mass implies representability while the two notions are equivalent if the current has compact support.  
Since the theory of currents has been first developed in the Riemannian setting,  understandably we refer to it as much as necessary to present concepts in a linear way. Specifically, we point out when some results can be compared to the Riemannian equivalent, citing the books of Federer (see Section 4.1 in \cite{FED}), Simon (\cite{SIMON}) and Morgan (see Chapter 4 in \cite{MORGAN1}). Another important reference is the 2007 work by Franchi, Serapioni and Serra Cassano (\cite{FSSC}).   \\
Currents, depending on their integration properties and on those of their boundaries, can be classified into subspaces. Particularly, in case we assume their support to be compact, we can work with currents of finite mass (see scheme below and figure \ref{fig:diagram_currents}); otherwise we need to consider currents with only locally finite mass (see figure \ref{fig:diagram_currents_lfm}).
\begin{figure*}[!ht]
\centering
\begin{empheq}[box=\fbox]{alignat*=8}
& \underset{\substack{\mathbb{H}\text{-reg. integral currents}}}{   I_{\mathbb{H},k}(U)  } & \subseteq & \underset{\substack{\mathbb{H}\text{-regular currents}}}{    \mathcal{R}_{\mathbb{H},k} (U) } & & \\
& \quad \quad \quad  \rotatebox[origin=c]{270}{$\subseteq$} & &  \quad \quad \quad   \rotatebox[origin=c]{270}{$\subseteq$}  & & \\
 & \underset{\substack{\mathbb{H}\text{-rect. integral currents}}}{I_{\mathbb{H}\text{-rect},k}(U)} & \subseteq & \underset{\substack{\mathbb{H}\text{-rectifiable currents}}}{\mathcal{R}_{\mathbb{H}\text{-rect},k}(U)} &  &   \\
&  \quad \quad \quad   \rotatebox[origin=c]{270}{$\subseteq$}  & & \quad \quad \quad    \rotatebox[origin=c]{270}{$\subseteq$}      & &     \\
&  \underset{\substack{\mathbb{H}\text{-normal currents}}}{   N_{\mathbb{H},k}(U) }     & \subseteq      &       \underset{\substack{\emph{currents with finite mass}}}{ R_{\mathbb{H},k}(U) }         &   \subseteq  &  \underset{\substack{\emph{currents with compact support}}}{ \mathcal{E}_{\mathbb{H},k}(U) }     \subseteq    \underset{\substack{\text{Rumin Currents}}}{ \mathcal{D}_{\mathbb{H},k}(U) }  
\end{empheq}
\end{figure*}

In Section \ref{section:slicing}, we define the notion of slices of Heisenberg currents and show some important properties for them. Slices are defined as follows:
\begin{defin_nonumber}[\ref{defin:slices}]
Consider an open set $U \subseteq \mathbb{H}^n$, $f \in Lip(U,\mathbb{R})$, $t \in \mathbb{R}$ and $T\in \mathcal{D}_{\mathbb{H},k}(U)$. We define \emph{slices of $T$} the following two currents:
\begin{align*}
\langle T,f,t+ \rangle :=\left ( \partial T \right ) \mres \{ f>t \} -  \partial \left ( T \mres \{ f>t \} \right ),\\
\langle  T,f,t - \rangle :=    \partial \left ( T \mres \{ f < t \} \right ) -  \left (  \partial T  \right ) \mres \{ f <t \}.
\end{align*}
\end{defin_nonumber}
In propositions \ref{first4properties} and \ref{next3properties}, we show seven properties for slices of Heisenberg currents. Specifically, Proposition \ref{first4properties} holds properties similarly true in Riemannian settings (compare with 4.2.1 in \cite{FED}) and we do not see an explicit use of the sub-Riemannian geometry in the proofs:
\begin{prop_nonumber}[\ref{first4properties}]
Consider an open set $U \subseteq \mathbb{H}^n$,  $T\in N_{\mathbb{H},k}(U)$, $f \in Lip(U,\mathbb{R})$, and $t \in \mathbb{R}$. Then we have the following properties:
\begin{enumerate}[(1)]
  \setcounter{enumi}{-1}
\item
$\left (  \mu_T + \mu_{\partial T}   \right ) \left (  \{ f=t \}  \right ) =0 \quad \quad \text{for all } t \text{ but at most countably many}.$
\item
$\langle T,f,t+ \rangle =\langle T,f,t- \rangle  \quad \quad \text{for all } t \text{ but at most countably many}.$
\item
$\spt \langle  T,f,t+ \rangle \subseteq f^{-1} \{t\}   \cap \spt T.$
\item
$ \partial  \langle  T,f,t+ \rangle   = -  \langle \partial  T,f,t+ \rangle    .$
\end{enumerate}
\end{prop_nonumber}
On the other hand, the proof of Proposition \ref{next3properties}, containing the remaining properties, is way more complex than in the Riemannian case and requires to explicitly work with the Rumin cohomology (see Lemma \ref{L2} in particular).
\begin{prop_nonumber}[\ref{next3properties}]
Consider an open set $U \subseteq \mathbb{H}^n$, $T\in N_{\mathbb{H},k+1}(U)$, $f \in Lip(U,\mathbb{R})$, $t \in \mathbb{R}$ and $k\neq n$. Then the following properties hold:
\begin{enumerate}[(1)]
  \setcounter{enumi}{3}
\item
$M \left ( \langle T,f,t+ \rangle \right ) \leq    \ Lip(f) \liminf\limits_{h \rightarrow 0+} \frac{1}{h} \mu_T \left (   U \cap \{ t < f < t+h \}  \right )$.
\item
$\int_a^b M \left ( \langle T,f,t+ \rangle \right ) dt \leq    \ Lip(f) \mu_T \left (   U \cap \{a < f < b \}  \right ), \quad a,b \in \mathbb{R}$.
\item
$ \langle  T,f,t+ \rangle \in N_{\mathbb{H},k}(U)  \quad  \text{for a.e. } t$.
\end{enumerate}
\end{prop_nonumber}
Proposition \ref{next3properties} carries deep consequences for the possibility of developing a compactness theorem for currents in the Heisenberg group because it does not include the slices of the middle dimension $k=n$, which opens new challenges and scenarios.  \\
Furthermore, this suggests that the study of currents on the first Heisenberg group $\mathbb{H}^1$ diverges from the other cases, because that is the only situation in which the dimension of the slice of a hypersurface, $2n-1$, coincides with the middle dimension $n$, which triggers a change in the associated differential operator in the Rumin complex.   %
Our future studies will focus, on one side, on the manipulation of the second order differential operator $D$ in the case of the first Heisenberg group $\mathbb{H}^1$ and, on the other side, on slices of currents with dimension different from $n$ for general $n\neq 1$. The case $k=n$ is also subject to ongoing research.\\\\
\noindent
{\bf Acknowledgments.}  
I would like to thank my adviser, university lecturer Ilkka Holopainen, for the work done together and the time dedicated to me.   
 I also want to thank professors Bruno Franchi, Raul Serapioni, Pierre Pansu and academy research fellow Tuomas Orponen for the stimulating discussions.


\section{Preliminaries}\label{sec:preliminaries}

In this section we introduce the Heisenberg group $\mathbb{H}^n$, its structure as a Carnot group
and the standard bases of vector fields and differential forms.   
There exist many good references for such an introduction and we follow mainly sections 2.1 and 2.2 in \cite{FSSC} and sections 2.1.3 and 2.2 in \cite{CDPT}.  
We also describe briefly the Rumin cohomology and complex; more detail descriptions can be found, for example, in \cite{RUMIN}, \cite{TRIP} and \cite{GClicentiate}.

\subsection{The Heisenberg Group $\mathbb{H}^n$}\label{defH}

\begin{defin}\label{Heisenberg_Group}             
The $n$-dimensional \emph{Heisenberg Group} $\mathbb{H}^n$ is defined as 
$
\mathbb{H}^n:= (\mathbb{R}^{2n+1}, * ),
$ 
where $*$ is the product
$$
(x,y,t)*(x',y',t') := \left  (x+x',y+y', t+t'- \frac{1}{2} \langle J
 \begin{pmatrix} 
x \\
y
\end{pmatrix} 
, 
 \begin{pmatrix} 
x' \\
y' 
\end{pmatrix} \rangle_{\mathbb{R}^{2n}} \right  ),
$$
with $x,y,x',y' \in \mathbb{R}^n$, $t,t' \in \mathbb{R}$ and $J=  \begin{pmatrix} 
 0 &  I_n \\
-I_n & 0
\end{pmatrix} $. 
It is common to write $x=(x_1,\dots,x_n) \in \mathbb{R}^n$. Furthermore, with a simple computation of the matrix product, we immediately have that
$$
(x,y,t)*(x',y',t') := \left  (x+x',y+y', t+t' + \frac{1}{2} \sum_{j=1}^n \left ( x_j y_j'  -  y_j x_j'  \right ) \right  ).
$$
\end{defin}

\noindent
One can verify that the Heisenberg group $\mathbb{H}^n$ is a Lie group, meaning that the internal operations of product and inverse are both differentiable.   
In the Heisenberg group $\mathbb{H}^n$ there are two important groups of automorphisms; the first one is the left translation 
\begin{align*}
\tau_q : \mathbb{H}^n  \to \mathbb{H}^n, \ p \mapsto q*p,
\end{align*}
and the second one is the ($1$-parameter) group of the \emph{anisotropic dilations $\delta_r$}, with $r>0$:
\begin{align*}
\delta_r : \mathbb{H}^n   \to \mathbb{H}^n, \ (x,y,t)  \mapsto (rx,ry,r^2 t).
\end{align*}

\noindent
On the Heisenberg group  $\mathbb{H}^n$ we can define different equivalent distances: the Kor\'anyi and the Carnot--Carath\'eodory distance.

\begin{defin}
\label{norm}
We define the \emph{Kor\'anyi} distance on $\mathbb{H}^n$ by setting, for $p,q \in \mathbb{H}^n$,
$$
d_{\mathbb{H}} (p,q) :=  \norm{ q^{-1}*p }_{\mathbb{H}},  
$$
where $ \norm{ \cdot }_{\mathbb{H}}$ is the \emph{Kor\'anyi}  norm
$  
\norm{(x,y,t)}_{\mathbb{H}}:=\left (  |(x,y)|^4+16t^2  \right )^{\frac{1}{4}},
$  
with $(x,y,t) \in \mathbb{R}^{2n} \times  \mathbb{R} $ and $| \cdot |$ being the Euclidean norm.
\end{defin}

\noindent
The Kor\'anyi distance is left invariant, meaning  
$ 
d_{\mathbb{H}} (p*q,p*q')=d_{\mathbb{H}} (q,q')$  for  $p,q,q' \in \mathbb{H}^n,
$
and homogeneous of degree $1$ with respect to $\delta_r$, meaning 
$
d_\mathbb{H} \left ( \delta_r (p), \delta_r (q)  \right ) = r d_{\mathbb{H}} (p,q) $,  for  $ p,q \in \mathbb{H}^n$ and $ r>0.
$

\noindent
Furthermore, the \emph{Kor\'anyi} distance is equivalent to the \emph{Carnot--Carath\'eodory} distance $d_{cc}$, which is measured along curves whose tangent vector fields are horizontal.


\subsection{Left Invariance and Horizontal Structure on $\mathbb{H}^n$}\label{lefthor}

The standard basis of vector fields in the Heisenberg group $\mathbb{H}^n$ gives it the structure of Carnot group. By duality, we also introduce its standard basis of differential forms.

\begin{defin}
\label{XYT}
The standard basis of left invariant vector fields in $\mathbb{H}^n$
consists of the following: 
 $$
\begin{cases}
X_j &:= \partial_{x_j} - \frac{1}{2} y_j\partial_{t} \quad \emph{\emph{ for }}  j=1,\dots,n , \\
Y_j &:= \partial_{y_j} + \frac{1}{2} x_j\partial_{t} \quad \emph{\emph{ for }}  j=1,\dots,n,  \\
T &:= \partial_{t}.
\end{cases}
 $$
\end{defin}

\noindent   
One can observe that $\{ X_1,\dots,X_n,Y_1,\dots,Y_n,T \}$ becomes $\{ \partial_{x_1},\dots, \partial_{x_n}, \partial_{y_1},\dots,\partial_{y_n}, \partial_{t} \}$ at the neutral element. Another easy observation is that the only non-trivial commutators of the vector fields $X_j,Y_j$ and $T$ are   
$
[X_j,Y_j]=T  $, 
for $j=1,\dots,n.
$   
This immediately tells that all higher-order commutators are zero and that the Heisenberg group is a Carnot group of step $2$. Indeed we can write its Lie algebra $\mathfrak{h}$ as 
$
\mathfrak{h} =\mathfrak{h}_1 \oplus \mathfrak{h}_2,
$ 
with
$$
\mathfrak{h}_1 = \spn \{ X_1,  \ldots, X_n, Y_1, \ldots, Y_n \} \quad \text{and} \quad \mathfrak{h}_2 =\spn \{ T \}.
$$
Conventionally one calls $\mathfrak{h}_1$ the space of \emph{horizontal} and $\mathfrak{h}_2$ the space of \emph{vertical vector fields}.   
The vector fields $\{ X_1,\dots,X_n,Y_1,\dots,Y_n\}$ are homogeneous of order $1$ with respect to the dilation $\delta_r, \  r \in \mathbb{R}^+$, i.e.,
$$
X_j (f\circ \delta_r)=r X_j(f)\circ \delta_r \quad  \text{ and }   \quad   Y_j (f\circ \delta_r)=r Y_j(f)\circ \delta_r ,
$$
where $f \in C^1 (U, \mathbb{R} )$, $U\subseteq \mathbb{H}^n$ open and $j=1,\dots,n$. On the other hand, the vector field $T$ is homogeneous of order $2$, i.e.,
$$
T(f\circ \delta_r)=r^2T(f)\circ \delta_r.
$$
It is not a surprise, then, that the homogeneous dimension of $\mathbb{H}^n$ is $Q=2n+2.$

\noindent
The vector fields $X_1,\dots,X_n,Y_1,\dots,Y_n,T$ form an orthonormal basis of $\mathfrak{h}$ with a scalar product $\langle \cdot , \cdot \rangle $. In the same way, $X_1,\dots,X_n,Y_1,\dots,Y_n$ form an orthonormal basis of $\mathfrak{h}_1$ with a scalar product $\langle \cdot , \cdot \rangle_H $ defined purely on $\mathfrak{h}_1$.

\begin{no}\label{notW}
Sometimes it will be useful to consider all the elements of the basis of $\mathfrak{h}$ with one symbol; to do so, we write
$$
\begin{cases}
W_j &:= X_j \quad \text{ for } j=1,\dots,n,\\
W_{n+j} &:= Y_j  \quad \text{ for } j=1,\dots,n,\\
W_{2n+1}&:=T.
\end{cases}
$$
In the same way, the point $(x_1,\dots,x_n,y_1,\dots,y_n,t)$ will be denoted as $(w_1,\dots,w_{2n+1})$.
\end{no}

\begin{defin}
\label{dual_basis}
Consider the dual space of $\mathfrak{h}$, $ {\prescript{}{}\bigwedge}^1 \mathfrak{h}$, which inherits an inner product from $\mathfrak{h}$. By duality, one can find a dual orthonormal basis of covector fields $\{\omega_1,\dots,\omega_{2n+1}\}$ in $ {\prescript{}{}\bigwedge}^1 \mathfrak{h}$ such that 
$$
\langle \omega_j \vert W_k \rangle =
\delta_{jk}, \quad \text{for } j,k=1,\dots,2n+1,
$$
where $W_k$ is an element of the basis of $\mathfrak{h}$. 
Such covector fields are differential forms in the Heisenberg group.
\end{defin}

\noindent
The orthonormal basis of $ {\prescript{}{}\bigwedge}^1 \mathfrak{h}$ is given by  
$
\{dx_1,\dots,dx_n,dy_1,\dots,dy_n,\theta \},
$  
where $\theta$ is called \emph{contact form} and is defined as
$$
\theta :=dt - \frac{1}{2}  \sum_{j=1}^{n} (x_j d y_j-y_j d x_j).
$$

\begin{ex}\label{df} 
As a useful example, we show here that the just-defined bases of vectors and covectors behave as one would expect when differentiating. Specifically, consider $f: U\subseteq \mathbb{H}^n \to \mathbb{R}$, $U$ open, $f \in C^1 (U, \mathbb{R} )$, then one has:
\begin{align*}
 df
=&  \sum_{j=1}^{n}  \left ( X_j f d x_j + Y_j f dy_j \right ) + Tf \theta.
\end{align*}
\end{ex}

\begin{defin} \label{kdim}   
We define the sets of $k$-dimensional vector fields and differential forms, respectively, as:
\begin{align*}
\Omega_k \equiv {\prescript{}{}\bigwedge}_k \mathfrak{h} &:= \spn \{ W_{i_1} \wedge \dots \wedge W_{i_k} \}_{1\leq i_1 \leq \dots \leq i_k \leq 2n+1 },
\end{align*}
and
\begin{align*}
\Omega^k \equiv {\prescript{}{}\bigwedge}^k \mathfrak{h} &:= \spn \{ dw_{i_1} \wedge \dots \wedge dw_{i_k} \}_{1\leq i_1 \leq \dots \leq i_k \leq 2n+1 }.
\end{align*}
The same definitions can be given for $ \mathfrak{h}_1$ and produce the spaces $ {\prescript{}{}\bigwedge}_k \mathfrak{h}_1 $ and $ {\prescript{}{}\bigwedge}^k \mathfrak{h}_1 $.
\end{defin}

\begin{defin}[see 2.3 in \cite{FSSC}]\label{def:article5-6_star}
Consider a form $\omega \in  {\prescript{}{}\bigwedge}^k \mathfrak{h}$, with $k=1,\dots,2n+1$. We define $\omega^* \in  {\prescript{}{}\bigwedge}_k \mathfrak{h}$ so that
$$
\langle \omega^* , V   \rangle   =     \langle \omega \vert V   \rangle \quad \text{for all } V \in  {\prescript{}{}\bigwedge}_k \mathfrak{h}.
$$
\end{defin}

%

\noindent
Next we give the definition of Pansu differentiability for maps between Carnot groups $\mathbb{G}$ and $\mathbb{G}'$. After that, we state it in the special case of $\mathbb{G}=\mathbb{H}^n$ and $\mathbb{G}'=\mathbb{R}$.\\
We call a function $h : (\mathbb{G},*,\delta) \to (\mathbb{G}',*',\delta')$ \emph{homogeneous} if $h(\delta_r(p))= \delta'_r \left ( h(p) \right )$ for all $r>0$.

\begin{defin}[see \cite{PANSU} and 2.10 in \cite{FSSC}]\label{dGGG}
Consider two Carnot groups $(\mathbb{G},*,\delta)$ and $(\mathbb{G}',*',\delta')$. A function $f: U \to \mathbb{G}'$, $U \subseteq \mathbb{G}$ open, is \emph{P-differentiable} at $p_0 \in U$ if there is a (unique) homogeneous Lie group 
 homomorphism $d_H f_{p_0} : \mathbb{G} \to \mathbb{G}'$ such that
$$
d_H f_{p_0} (p) := \lim\limits_{r \to 0} \delta'_{\frac{1}{r}} \left ( f(p_0)^{-1} *' f(p_0* \delta_r (p) ) \right ),
$$
uniformly for $p$ in compact subsets of $U$.
\end{defin}

\begin{defin}\label{dHHH} 
Consider a function $f: U \to \mathbb{R}$, $U \subseteq \mathbb{H}^n$ open. $f$ is \emph{P-differentiable} at $p_0 \in U$ if there is a (unique) homogeneous Lie group 
 homomorphism $d_H f_{p_0} : \mathbb{H}^n \to \mathbb{R}$ such that
$$
d_H f_{p_0} (p) := \lim\limits_{r \to 0} \frac{  f \left (p_0* \delta_r (p) \right ) - f(p_0) }{r},
$$
uniformly for $p$ in compact subsets of $U$.
\end{defin}


\begin{defin}[see 2.11 in \cite{FSSC}]\label{veryfirstnabla} 
Consider a function $f$ P-differentiable  at $p \in U$, $f:U \to \mathbb{R}$, $U\subseteq \mathbb{H}^n$ open. 
The \emph{Heisenberg gradient} or \emph{horizontal gradient} of $f$ at $p$ is defined as
$$
\nabla_\mathbb{H} f(p) := \left ( d_H f_p \right )^* \in \mathfrak{h}_1,
$$
or, equivalently,
$$
\nabla_\mathbb{H} f(p) = \sum_{j=1}^{n} \left [  (X_j f)(p) X_j  + (Y_j f)(p) Y_j  \right ].
$$
\end{defin}

\begin{no}[see 2.12 in \cite{FSSC}]\label{CH1}
Sets of differentiable functions can be defined with respect to the P-differentiability. Consider $ U \subseteq \mathbb{G}$ and $V \subseteq \mathbb{G}'$ open, then  
$C_{\mathbb{H}}^1 (U, V)$ is the vector space of continuous functions $f:U \to V $  such that the P-differential $d_H f$ is continuous.
\end{no}

\noindent
To conclude this part, we define the Hodge operator which, given a vector field, returns a second one of dual dimension  
and orthogonal to the first.

\begin{defin}[see 2.3 in \cite{FSSC} or 1.7.8 in \cite{FED}]\label{hodge}
Consider $1 \leq k \leq 2n$. The \emph{Hodge operator} is the linear isomorphism
\begin{align*}
*: {\prescript{}{}\bigwedge}_k \mathfrak{h} &\rightarrow {\prescript{}{}\bigwedge}_{2n+1-k} \mathfrak{h} ,\\
\sum_I v_I V_I &\mapsto  \sum_I v_I (*V_I),
\end{align*}
where 
$
*V_I:=(-1)^{\sigma(I) }V_{I^*},
$ 
and, for $1 \leq  i_1 \leq \cdots \leq i_k \leq 2n+1$,
\begin{itemize}
\item $I=\{ i_1,\cdots,i_k \}$,
\item $V_I= V_{i_1} \wedge \cdots \wedge V_{i_k} $,
\item $I^*=\{ i_1^*,\dots,i_{2n+1-k}^* \}=\{1, \cdots, 2n+1\} \smallsetminus I \quad $  and
\item $\sigma(I)$ is the number of couples $(i_h,i_l^*)$ with $i_h > i_l^*$.
\end{itemize}
\end{defin}


\subsection{Rumin Cohomology in $\mathbb{H}^n$}

The Rumin cohomology is the equivalent of the Riemann cohomology but for the Heisenberg group. Its complex is given not by one but by three operators, depending on the dimension.

\begin{defin}\label{def_forms}  
Consider $0\leq k \leq 2n+1$ and recall $\Omega^k$ from Definition \ref{kdim}. We denote:
\begin{itemize}
\item
$I^k := \{ \alpha \wedge \theta + \beta \wedge d \theta ; \  \alpha \in \Omega^{k-1}, \ \beta \in \Omega^{k-2}  \}$,
\item
$J^k :=\{ \alpha \in \Omega^{k}; \  \alpha \wedge \theta =0, \  \alpha \wedge d\theta=0   \}$.
\end{itemize}
\end{defin}

\begin{no}[see 2.1.8 and 2.1.10 in \cite{GClicentiate}]\label{notationL}
We denote $L$ the operator
\begin{align*}
L:  {\prescript{}{}\bigwedge}^{n-1} \mathfrak{h}_1  \to {\prescript{}{}\bigwedge}^{n+1} \mathfrak{h}_1, \  \beta  \mapsto d \theta \wedge \beta .
\end{align*}
Furthermore we remind that, if $\gamma \in \Omega^{k-1}$, we can consider the equivalence class
$$
  {\prescript{}{}\bigwedge}^k \mathfrak{h}_1  =       \left \{  \beta \in \Omega^k ; \ \beta =0 \ \text{or} \ \beta \wedge \theta \neq 0 \right  \}  \cong  \frac{\Omega^k}{ \{ \gamma \wedge \theta \} }    ,
$$
where we write $\{ \gamma \wedge \theta \} = \{ \gamma \wedge \theta ; \ \gamma \in \Omega^{k-1} \}$ for short. The equivalence is given by $\beta \mapsto ( \beta)_{\vert_{ {\prescript{}{}\bigwedge}^{k} \mathfrak{h}_1 }}$.\\
In particular, $L$ is an isomorphism (see $2$ in \cite{RUMIN}) and we can denote
$$
\mathcal{L} (\alpha) :=  L^{-1}   \left ( - \left ( d  \alpha \right )_{\vert_{ {\prescript{}{}\bigwedge}^{n+1} \mathfrak{h}_1 }}  \right ) .
$$
%
%
\end{no}

\begin{no}\label{notation:IJ}
We denote by $ [\alpha]_{I^k}$ an element of the quotient $\frac{\Omega^k}{I^k}$ and  $\omega_{\vert_{ J^k }}$ an element of $J^k$ whenever  $\omega \in   \mathcal{D}^{k} (U)$. We will use this second definition later on.
\end{no}

\begin{defin}[Rumin complex]\label{complexHn}    
The Rumin complex, due to Rumin in \cite{RUMIN}, is given by
$$
0 \to \mathbb{R} \to C^\infty  \stackrel{d_Q}{\to} \frac{\Omega^1}{I^1}  \stackrel{d_Q}{\to}  \dots \stackrel{d_Q}{\to} \frac{\Omega^n}{I^n} \stackrel{D}{\to} J^{n+1}    \stackrel{d_Q}{\to} \dots   \stackrel{d_Q}{\to} J^{2n+1} \to 0,
$$
where $d$ is the standard differential operator and, for $k < n$,
$$
d_Q( [\alpha]_{I^k} ) :=  [d \alpha]_{I^{k+1}},
$$
while, for $k \geq n +1$,
$$
d_Q := d_{\vert_{J^k}}.
$$
The second order differential operator $D$ is defined as
$$
D( [\alpha]_{I^n} ) :=  d \left ( \alpha +  L^{-1} \left (- (d \alpha)_{\vert_{ {\prescript{}{}\bigwedge}^{n+1} \mathfrak{h}_1 }} \right ) \wedge \theta \right )
=  d \left ( \alpha +  \mathcal{L} (\alpha)  \wedge \theta \right ).
$$
These three different differential operators are at times denoted with the same syntax $d_c$ or $d_c^{(k)}$, when they act on $k$-forms (see Theorem 11.40 in \cite{TRIP} or Proposition B.7 in \cite{GClicentiate}).
\end{defin}


\section{Currents in the Heisenberg Group}\label{currents}

In this section we first define the notion of current in the Heisenberg group and expose its relationship with Riemannian currents.   
Then we describe how currents can be written as integrals with the notion of \emph{representability by integration}, define the mass of a current in $\mathbb{H}^n$ and show that finite mass implies representability and the two notions are equivalent if the current has compact support.   
Last, we classify currents into subspaces depending on the integration properties of themselves and their boundaries and 
 we work with currents with finite mass if the support is compact (see figure \ref{fig:diagram_currents}), while we 
 consider currents with only locally finite mass otherwise (see figure \ref{fig:diagram_currents_lfm}).  
In Riemannian geometry there are different kind of currents and the correlation between the different definitions is well known since Federer (see Section 4.1 in \cite{FED}); useful references are also the works of Simon (\cite{SIMON}) and Morgan (see Chapter 4 in \cite{MORGAN1}).   
Finally, for the Heisenberg group specifically, an important reference is the 2007 work by Franchi, Serapioni and Serra Cassano (\cite{FSSC}).

\begin{defin}[see 
5.8 in \cite{FSSC}]\label{def:RuminCurrents}
Consider an open set $U \subseteq \mathbb{H}^n$. 
We call $\mathcal{D}_{\mathbb{H}}^k (U)$ the space of compactly supported smooth sections on $U$ of, respectively, $\frac{\Omega^k}{I^k}$, if $1 \leq k \leq n$, and $J^{k}$, if $n+1 \leq k \leq 2n+1$. %
These spaces are topologically locally convex. %
For convenience, we call the elements of $\mathcal{D}_{\mathbb{H}}^k (U)$ \emph{Rumin} or \emph{Heisenberg differential forms}.\\
Furthermore, we call \emph{Rumin} or \emph{Heisenberg current} any continuous linear functional from the space $\mathcal{D}_{\mathbb{H}}^k (U)$ to $\mathbb{R}$ and we denote their set as $\mathcal{D}_{\mathbb{H},k} (U).$
\end{defin}

\noindent
We just saw in Definition \ref{def:RuminCurrents} that the Rumin currents are defined, for low dimensions, on quotient spaces. Nevertheless it is possible, to fix the ideas, to think about Rumin differential forms as a subset of the standard differential forms and so write $\mathcal{D}_{\mathbb{H}}^k(U)  \subseteq  \mathcal{D}^k(U)$ for simplicity. In the same way, we can think about Rumin currents as a subset of the Euclidean currents. Indeed, any Rumin current $T \in \mathcal{D}_{\mathbb{H},k} (U)$ can be identified with an Euclidean $k$-current $\widetilde T \in  \mathcal{D}_{k} (U)$ by setting, for $ \omega \in \mathcal{D}^{k} (U)$:\\
$$
\widetilde T (\omega) :
=
\begin{cases}
 T([\omega]_{I^k} ), 
\quad  \text{where } [\omega]_{I^k} \in \mathcal{D}_{\mathbb{H}}^{k} (U) =\frac{\Omega^k}{I^k},
\quad \text{if }1 \leq k \leq n,\\
 T \left (   \omega_{\vert_{ J^k }}    \right ),
\quad  \text{where }    \omega = \omega_{\vert_{ \left (J^k  \right )^\perp }}   +   \omega_{\vert_{ J^k }},
    \omega_{\vert_{ \left (J^k  \right )^\perp }} \notin  J^k  \text{ and }   \omega_{\vert_{ J^k }} \in   \mathcal{D}_{\mathbb{H}}^{k} (U)= J^k ,\\      
  \quad \quad \quad \quad \quad \quad \quad \quad \quad \quad \quad \quad \quad \quad \quad \quad \quad \quad \quad \quad \quad \quad \text{if } n+1 \leq k \leq 2n+1.
\end{cases}
$$

\begin{defin}[compare with 4.1.1 in \cite{FED}]\label{supportCurrent}
Consider an open set $U \subseteq \mathbb{H}^n$ and $T \in \mathcal{D}_{\mathbb{H}}^k (U)$. The \emph{support of a current T} is defined as
$$
\spt T :=  U \setminus \bigcup \{  V : V\subseteq U, \ V \text{ open}, \ T(\omega)=0 \text{ for all } \omega \in \mathcal{D}^k_{\mathbb{H}}  (U), \ \spt \omega \subseteq V   \},
$$
where $\spt \omega = \overline{\{ x\in U \ / \ \omega(x)\neq 0 \}}$. 

\end{defin}

\subsection{Representability by Integration and Masses in $\mathbb{H}^n$}
In the study of currents, it is often useful to be able to write a current as an integral. The first notion we see that allows us to do so is 
 \emph{representability by integration}. After that we define the mass of currents in $\mathbb{H}^n$ and show that finite mass implies representability and the two notions are equivalent if the current has compact support.   \\
Since the theory of currents has been first developed in the Riemannian setting,  understandably we refer to it as much as necessary to present concepts in a linear way. Specifically, we point out when some results can be compared to the Riemannian equivalent, citing the books of Federer (\cite{FED}), Simon (\cite{SIMON}) and Morgan (\cite{MORGAN1}). Another important reference is the 2007 work by Franchi et al. (\cite{FSSC}).

\begin{defin}[see 
 2.5 in \cite{FSSC}]\label{HLambda}
Recall Definitions \ref{kdim} and \ref{hodge}. For $0\leq k \leq n$, we denote
\begin{align*}
{\prescript{}{H}\bigwedge}_0& := \mathbb{R} ,\\
{\prescript{}{H}\bigwedge}_k &:= \left \{  v \in {\prescript{}{}\bigwedge}_k \mathfrak{h}_1 \ / \ v \text{ simple and integrable}     \right \},\\
{\prescript{}{H}\bigwedge}_{2n+1-k}  &:= *  \left (  {\prescript{}{H}\bigwedge}_k   \right ),
\end{align*}
where $v$ is \emph{integrable} if and only if the distribution associated to it is so. By duality, for $0\leq k \leq 2n+1$,
\begin{align*}
{\prescript{}{H}\bigwedge}^k := {\prescript{}{}\bigwedge}^1  \left (    {\prescript{}{H}\bigwedge}_k  \right ) = \left \{    \omega \in {\prescript{}{}\bigwedge}^k \mathfrak{h} \ / \ \omega^* \in {\prescript{}{H}\bigwedge}_k   \right \}.
\end{align*}
Note that, by Theorem 2.9 in \cite{FSSC}, the spaces ${\prescript{}{H}\bigwedge}^k $'s are the spaces of the Rumin cohomology. So the spaces of vector fields ${\prescript{}{H}\bigwedge}_k$'s are the dual of the Rumin differential forms.
\end{defin}

\begin{defin}\label{repbyint}
Consider an open set $U \subseteq \mathbb{H}^n$ and $T\in \mathcal{D}_{\mathbb{H},k} (U)$. We say that $T$ is \emph{ representable by integration}, and we write $T=\overrightarrow{T} \wedge \mu_T$, if there exist $\mu_T$ a Radon measure over $U$ and a vector $\overrightarrow{T} : U \to {\prescript{}{H}{\bigwedge}}_k$ $\mu_T$-meas. s.t. 
\begin{align*}
\begin{aligned}
\norm{\overrightarrow{T}(p)}=1 \quad &\text{ for } \mu_T\text{-a.a. } p \in U \quad \text{ and }\\
T(\omega)= \int \langle \omega(p) \vert \overrightarrow{T}(p) \rangle d \mu_T(p) \ \ \quad &\text{ for all }  \omega \in \mathcal{D}_{\mathbb{H}}^k (U). 
\end{aligned}
\end{align*}
%
\end{defin}





\noindent
Before we define the mass of a current, a clarification is necessary. 
In the standard theory of currents there are two different notion of \emph{mass} for a current: one made using the comass of differential forms (see 4.3 in \cite{MORGAN1} and 4.1.7 in \cite{FED}) and one using the norm given by the inner product of differential forms (see, for instance, 
2.6Ch6 in \cite{SIMON}). 
This is still true in our case.



\begin{defin}[mass of a current by the comass in $\mathbb{H}^n$]\label{masscomass}
Consider an open set $U \subseteq \mathbb{H}^n$ and $T \in \mathcal{D}_{\mathbb{H},k} (U)$. Denote the mass of a current $T$ defined by the comass as: 
$$
M(T)  := \sup    \left     \{    T(\omega) , \  \omega \in  \mathcal{D}_{\mathbb{H}}^k(U), \ \norm{\omega(p)}^* \leq 1  \ \forall p \right   \}  = \sup\limits_{\omega \in  \mathcal{D}_{\mathbb{H}}^k(U), \ \norm{\omega(p)}^* \leq 1 }       T(\omega)  ,
$$
with \emph{comass}
\begin{align*}
\norm{\omega (p)}^*   :&= \sup     \left     \{  
 \langle  \omega (p) \vert v \rangle  
 \ / \ v \text{ a unit, simple, integrable } k\text{-vector}         \right   \}  \\
&= \sup \left  \{       \omega_p (v) \ / \ v \in   {\prescript{}{H}\bigwedge}_k, \ \vert v \vert \leq 1  \right  \}.
\end{align*}
\end{defin}

\noindent
Other notations for the comass in the literature are $M(\omega)$ and $\norm{\omega(p)}$.

\begin{defin}[mass of a current by the scalar product in $\mathbb{H}^n$, see 
5.12 in \cite{FSSC}]\label{massFSSC}
Consider an open set $U \subseteq \mathbb{H}^n$ and $T \in \mathcal{D}_{\mathbb{H},k} (U)$. Denote the mass of a current $T$ defined by the scalar product as:  
$$
m(T):= \sup  \left  \{ T(\omega) \ / \ \omega \in \mathcal{D}_{\mathbb{H}}^k(U), \vert \omega \vert \leq 1 \right \}
$$
with $\vert \omega \vert = \sqrt{ \langle \omega , \omega \rangle}$, where $\langle \cdot , \cdot \rangle$ is the Riemannian scalar product that makes the differential forms $dx_j, dy_j$'s and $\theta$ orthonormal. 
\end{defin}

%

\noindent
The comass is smaller or equal than the scalar product norm (see also 2.6Ch6 in \cite{SIMON}), which means that the mass defined with the comass is bigger or equal than the one defined with the scalar product:
$$
m(T)  \leq M(T) \quad  \text{for all }  T \in \mathcal{D}_{\mathbb{H},k}(U).
$$

\noindent
Finally we state the correlation between mass and currents representable by integration (compare with 4.1.7 in \cite{FED} and 2.8Ch6 in \cite{SIMON}).

\begin{prop}\label{masstoint}
Consider an open set $U \subseteq \mathbb{H}^n$ and $T \in \mathcal{D}_{\mathbb{H},k} (U)$. Then
$M(T) < \infty$ implies that $T=\overrightarrow{T} \wedge \mu_T$  and  $\mu_T(U)=M(T).$
%
\end{prop}

\noindent
The proof is based on Riesz Representation Theorem and it is not dissimilar from the same proof in the Riemannian setting.


\begin{cor}\label{masses222}
Consider an open set $U \subseteq \mathbb{H}^n$ and $T \in \mathcal{D}_{\mathbb{H},k} (U)$. Then $m(T) < \infty$ implies that
 $T=\overrightarrow{T_m} \wedge  \mu_{T,m}$  and    $\mu_{T,m}(U)=m(T)$, where $\mu_{T,m}$ is the Radon measure relative to the mass $m$.\\
In particular (compare with 2.6Ch6 and 4.14Ch1 in \cite{SIMON}), if $M(T) < \infty$, then both masses are finite, $\mu_T$ is unique,  
 $\overrightarrow{ T}= \overrightarrow{T_m} \text{ a.e.}$  and  $\mu_T(U)=M(T)=m(T)= \mu_{T,m}(U)$.
\end{cor}


\begin{cor}[compare with 4.1.7 in \cite{FED}]\label{mass=T}
Consider an open set $U \subseteq \mathbb{H}^n$ and $T \in \mathcal{D}_{\mathbb{H},k} (U)$. If $\spt T$ is compact, then
$$
M(T) < \infty \quad \text{if and only if} \quad T=\overrightarrow{T} \wedge \mu_T.
$$
\end{cor}

\begin{proof}
From Proposition \ref{masstoint} we know that $M(T) < \infty$ implies $T=\overrightarrow{T} \wedge \mu_T$. On the other hand,
\begin{align*}
\left \vert
M(T)
\right \vert
\leq
 \sup\limits_{\omega \in  \mathcal{D}_{\mathbb{H}}^k(U), \ \norm{\omega(p)}^* \leq 1 }   
\left \vert
 \int \langle \omega \vert \overrightarrow{T} \rangle d \mu_T
\right \vert
\leq
 \int  \norm{\omega}^*  \norm{\overrightarrow{T}}  d \mu_T
\leq
\mu_T (U) < \infty
\end{align*}
because $T$ has compact support.
\end{proof}



\subsection{Classification of Sub-Riemannian Currents in $\mathbb{H}^n$}

Currents, depending on their integration properties and on those of their boundaries, can be classified into subspaces. Particularly, in case we assume their support to be compact, we can work with currents of finite mass (see figure \ref{fig:diagram_currents}); otherwise we need to consider currents with only locally finite mass (see figure \ref{fig:diagram_currents_lfm}).

\begin{defin}[see 5.19 in \cite{FSSC}]
Consider an open set $U \subseteq \mathbb{H}^n$, a current  $T \in  \mathcal{D}_{\mathbb{H},k}(U)$ and $1\leq k\leq 2n+1$. %
We call \emph{Heisenberg boundary of} $T$ the $(k-1)$-dimensional Heisenberg current denoted $\partial T$ (or sometimes $\partial_{\mathbb{H}} T$) and defined as:
\begin{align*}
\partial T (\omega) &:= T(d_Q \omega ), \quad \text{ if } k \neq n+1 \\ 
\end{align*}
and 
\begin{align*}
\partial T (\omega) &:= T(D \omega ), \quad \text{ if } k = n+1, \\
\end{align*}
where $\omega \in \mathcal{D}_{\mathbb{H}}^{k-1}(U)$.
\end{defin}

\begin{defin}
Consider an open set $U \subseteq \mathbb{H}^n$ and $1\leq k\leq 2n+1$. We define the \emph{space of currents with compact support} as
\begin{align*}
\mathcal{E}_{\mathbb{H},k} (U)  &  := \left  \{  T \in  \mathcal{D}_{\mathbb{H},k}(U)  \ / \ \spt T \text{ compact }  \right  \}.
\end{align*}
Furthermore, we can define the \emph{spaces of currents with finite mass} as
\begin{align*}
R_{\mathbb{H},k}(U)      :& = \left   \{   T \in   \mathcal{E}_{\mathbb{H},k} (U)   \ / \  M(T) < \infty         \right   \};  \\
  N_{\mathbb{H},k}    (U)       :&=       \left   \{     T \in \mathcal{E}_{\mathbb{H},k}(U)  \ / \     M(T)+ M( \partial T ) < \infty  \right   \} \subseteq R_{\mathbb{H},k}(U)     
.
\end{align*}
\end{defin}

\noindent
By Corollary \ref{mass=T}, we can immediately characterise the spaces as follows:
\begin{align*}
R_{\mathbb{H},k}(U)    & = \{ T \in \mathcal{E}_{\mathbb{H},k}(U)  \ / \  T=\overrightarrow{T} \wedge \mu_T \};   
\\
  N_{\mathbb{H},k}    (U)    &  =   \{  T \in \mathcal{E}_{\mathbb{H},k}(U)  \ / \  T=\overrightarrow{T} \wedge \mu_T, \ \partial T=\overrightarrow{\partial T} \wedge \mu_{ \partial T} \}.
\end{align*}

\noindent
The next step consists in defining rectifiable currents. For that we need to first define $\mathbb{H}$-regular and $\mathbb{H}$-rectifiable sets:

\begin{defin}[see 3.1 in \cite{FSSC}]
Consider $1\leq k \leq n$. A subset $S \subseteq \mathbb{H}^n$ is a $\mathbb{H}$-\emph{regular} $k$-\emph{dimensional surface} if for all $p \in S$   there exists  a neighbourhood $U$ of $p$,   
 an open set $ V \subseteq \mathbb{R}^k$ and a function $\varphi : V \to U$, $ \varphi \in C_{\mathbb{H}}^1(V,U) $  injective with $d_H \varphi $ injective, such that $ S \cap U = \varphi (V)$.
\end{defin}

\begin{defin}[see 3.2 in \cite{FSSC}]\label{Hreg}
Consider $1\leq k \leq n$. A subset $S \subseteq \mathbb{H}^n$ is a $\mathbb{H}$-\emph{regular} $k$-\emph{codimensional surface} if for all $ p \in S $ there exists a neighbourhood  $U$ of $p$   
 and a function  $ f : U \to \mathbb{R}^k$, $ f \in C_{\mathbb{H}}^1(U,\mathbb{R}^k)$, such that  $  {\nabla_\mathbb{H} f_1} \wedge \dots \wedge {\nabla_\mathbb{H} f_k}   \neq 0 $ on   $ U $ and  $  S \cap U = \{ f=0 \} $.
\end{defin}

\begin{defin}[see 
 5.1 in \cite{FSSC}]\label{rectset}
Consider $S \subseteq \mathbb{H}^n$ and $\mathcal{S}^k_\infty$ the spherical Haussdorff measure defined in Subsection $2.1$ in  \cite{FSSC}. We say that $S$ is a $k$\emph{-dimensional} $\mathbb{H}$\emph{-rectifiable set} if
$$
S \subseteq S_0 \cup \bigcup_{j=1}^{\infty} S_j,
$$
where
\begin{itemize}
\item
if $1 \leq k \leq n$: $S$ is $\mathcal{S}^k_\infty$-measurable, $\mathcal{S}^k_\infty (S) < \infty$, $S_j$'s are $k$-dimensional $\mathbb{H}$-regular surfaces and $\mathcal{S}^k_\infty (S_0) =0$;
\item
if $n+1 \leq k \leq 2n+1$: $S$ is $\mathcal{S}^{k+1}_\infty$-measurable, $\mathcal{S}^{k+1}_\infty (S) < \infty$, $S_j$'s are $(2n+1-k)$-dimensional $\mathbb{H}$-regular surfaces and $\mathcal{S}^{k+1}_\infty (S_0) =0$.
\end{itemize}

\end{defin}

\noindent
If $M \subseteq \mathbb{H}^n$ is a $\mathbb{H}$-rectifiable set, we can assume that (see 
5.7 in \cite{FSSC})
$$
M = M_0 \cup \bigcup_{j=1}^{\infty} M_j
$$
where $M_0$ has measure zero and $M_j$'s are pairwise disjointed Borel subsets of $\mathbb{H}$-regular surfaces $S_j$'s as in Definition \ref{rectset}. This implies that $M$ can be oriented by the $M_j$'s, when such orientations exist, up to the set $M_0$.   
Now we can define the set of rectifiable currents:

\begin{defin}\label{defrectH}
Consider an open set $U \subseteq \mathbb{H}^n$ and $1\leq k\leq 2n+1$. We define the \emph{space of $\mathbb{H}$-rectifiable currents} as
\begin{align*}
\mathcal{R}_{\mathbb{H}\text{-rect},k} (U) &   :=       \left   \{    T \in \mathcal{E}_{\mathbb{H},k} (U) \ / \      T(\omega) = \int_{U_T} \langle \omega (p) \vert \overrightarrow T (p) \rangle \rho (p)  d  \mu_k , \ \omega  \in \mathcal{D}_{\mathbb{H}}^k (U) \right   \} 
\end{align*}
where  
$U_T$ is an $\mathbb{H}$-rectifiable $k$-dimensional set oriented (up to a set of measure zero) by $\overrightarrow T $, a $\mu_k$-a.e. unit $k$-vector in ${\prescript{}{H}\bigwedge}^k$, $\rho$ is a positive integer multiplicity s.t. $\int_{U_T \cap \spt T} \rho (p) d  \mu_k < \infty $ and 
$$
\mu_k:=
\begin{cases}
\mathcal{S}^k_\infty, & \text{if } 1\leq k \leq n;\\
\mathcal{S}^{k+1}_\infty, & \text{if }  n+1\leq k \leq 2n+1.
\end{cases}
$$
Then we define the \emph{space of integral $\mathbb{H}$-rectifiable currents} as
\begin{align*}
  I_{\mathbb{H}\text{-rect},k} (U)     :=       \left   \{    T \in \mathcal{R}_{\mathbb{H}\text{-rect},k} (U) \ / \    \partial T \in \mathcal{R}_{\mathbb{H}\text{-rect},k-1} (U)      \right \}   \subseteq \mathcal{R}_{\mathbb{H}\text{-rect},k} (U) .
\end{align*}
\end{defin}

\begin{prop}\label{massfinimplies}
Consider an open set $U \subseteq \mathbb{H}^n$ and $1\leq k\leq 2n+1$. Then $T \in \mathcal{R}_{\mathbb{H}\text{-rect},k}  (U)$ implies $M (T) < \infty$, i.e.,
$$
\mathcal{R}_{\mathbb{H}\text{-rect},k}  (U)   \subseteq    R_{\mathbb{H},k}  (U) .
$$
This also immediately implies that $I_{\mathbb{H}\text{-rect},k}  (U)   \subseteq    N_{\mathbb{H},k}  (U) .$
\end{prop}

\begin{proof}
The proof is a simple computation. Consider $ T \in \mathcal{R}_{\mathbb{H}\text{-rect},k}  (U)$, then:
\begin{align*}
M(T)
&=
 \sup\limits_{\omega \in  \mathcal{D}_{\mathbb{H}}^k(U), \ \norm{\omega(p)}^* \leq 1 }   
\left \vert
\int_{U_T} \langle \omega (p) \vert \overrightarrow T (p) \rangle \rho (p)  d  \mu_k 
\right \vert\\
&\leq
 \sup\limits_{\omega \in  \mathcal{D}_{\mathbb{H}}^k(U), \ \norm{\omega(p)}^* \leq 1 }   
 \int_{U_T}  \norm{\omega}^*  \norm{\overrightarrow{T}} 
\rho (p) 
 d \mu_k\\
&\leq
 \int_{U_T \cap \spt T} \rho  (p)  d \mu_k
< \infty
\end{align*}
\end{proof}

\begin{defin}
Consider an open set $U \subseteq \mathbb{H}^n$ and $1\leq k\leq 2n+1$. In a similar way as above, we can define the \emph{spaces of $\mathbb{H}$-regular currents} and \emph{integral $\mathbb{H}$-regular currents} respectively as
\begin{align*}
 \mathcal{R}_{\mathbb{H},k}   (U)   &   :=       \left   \{    T \in \mathcal{R}_{\mathbb{H}\text{-rect},k}(U)  \ / \   U_T \text{ is an orientable } \mathbb{H}\text{-regular surface}           \right \} \quad \text{and}
\\
 I_{\mathbb{H},k}(U)  &   :=       \left   \{      T \in    \mathcal{R}_{\mathbb{H},k}(U)  \ / \        \partial T \in     \mathcal{R}_{\mathbb{H},k-1}    (U)      \right \}   \subseteq    \mathcal{R}_{\mathbb{H},k}  (U)   .
\end{align*}
\end{defin}

\noindent
Consider an open set $U \subseteq \mathbb{H}^n$ and $1\leq k\leq 2n+1$. By the definition it is straightforward that
\begin{align*}
\mathcal{R}_{\mathbb{H},k}  (U)   \subseteq \mathcal{R}_{\mathbb{H}\text{-rect},k}  (U) \quad \text{and} \quad  I_{\mathbb{H},k}  (U)   \subseteq    I_{\mathbb{H}\text{-rect},k}  (U)   .
\end{align*}


\begin{prop}[compare with Section 4.3B in \cite{MORGAN1}]
Consider an open set $U \subseteq \mathbb{H}^n$ and $1\leq k\leq 2n+1$. If $T\in \mathcal{R}_{\mathbb{H}\text{-rect},k}(U)$, we have that
$$
\mu_T(U)=   M(T) 
= \int_{U_T \cap \spt T} \rho  (p)  d \mu_k.
$$
\end{prop}

\begin{proof}
The first equality in the statement comes from Proposition \ref{masstoint}. For the second equality, by Proposition \ref{massfinimplies}, we know that $T\in \mathcal{R}_{\mathbb{H}\text{-rect},k}(U)$ implies  $T=\overrightarrow{T} \wedge \mu_T$. At the same time, $T\in \mathcal{R}_{\mathbb{H}\text{-rect},k}(U)$ says that we can write 
$$
T(\star)=\int_{U_T} \langle \star \vert \overrightarrow{T}  \rangle \rho  d  \mu_k = \overrightarrow{T} \wedge  \rho  \mu_k \mres \left (  U_T \cap \spt T\right ) (\star).
$$
By uniqueness of the representation by integration, that comes from Riesz Representation Theorem, we have that
$$
\mu_T = \rho  \mu_k \mres \left (  U_T \cap \spt T\right ), \quad \text{i.e.,} \quad \mu_T (U) =\int_{U_T \cap \spt T}  \rho (p) d \mu_k .
$$
\end{proof}

\noindent
We remind that a $C^1$-Euclidean regular $k$-surface can be written as $S=C(S) \cup \left (     S \setminus   C(S)   \right )$ where, for $n+1 \leq k \leq 2n+1$, $\mathcal{S}^{k+1}_\infty (C(S)) =0$ and  $S \setminus   C(S)  $ is a $\mathbb{H}$-regular surface (see page 195 in \cite{FSSC}).\\
For this reason, when  $n+1 \leq k \leq 2n+1$,  
\begin{align*}
 \mathcal{R}_{\mathbb{H},k}  (U)    &=   
    \left   \{    T \in \mathcal{R}_{\mathbb{H}\text{-rect},k} (U)  \ / \      U_T \text{ is an orientable } \mathbb{H}\text{-regular surface}     \right \} 
\\
  &   \supseteq       \left   \{    T \in \mathcal{R}_{\mathbb{H}\text{-rect},k} (U)  \ / \      U_T \text{ is an orientable }  C^1\text{-regular surface}    \right \} 
\end{align*}
and the same is true for $ I_{\mathbb{H},k}(U)$,

\noindent
The inclusions noted so far are summarised in figure \ref{fig:diagram_currents} (compare with 4.1.24 in \cite{FED}).\\

\begin{figure}[!ht]
\centering
\begin{empheq}[box=\fbox]{alignat*=8}
& \underset{\substack{\mathbb{H}\text{-reg. integral currents}}}{   I_{\mathbb{H},k}(U)  } & \subseteq & \underset{\substack{\mathbb{H}\text{-regular currents}}}{    \mathcal{R}_{\mathbb{H},k} (U) } & & \\
& \quad \quad \quad  \rotatebox[origin=c]{270}{$\subseteq$} & &  \quad \quad \quad   \rotatebox[origin=c]{270}{$\subseteq$}  & & \\
 & \underset{\substack{\mathbb{H}\text{-rect. integral currents}}}{I_{\mathbb{H}\text{-rect},k}(U)} & \subseteq & \underset{\substack{\mathbb{H}\text{-rectifiable currents}}}{\mathcal{R}_{\mathbb{H}\text{-rect},k}(U)} &  &   \\
&  \quad \quad \quad   \rotatebox[origin=c]{270}{$\subseteq$}  & & \quad \quad \quad    \rotatebox[origin=c]{270}{$\subseteq$}      & &     \\
&  \underset{\substack{\mathbb{H}\text{-normal currents}}}{   N_{\mathbb{H},k}(U) }     & \subseteq      & \quad \quad     R_{\mathbb{H},k}(U)        &   \subseteq  &\quad  \mathcal{E}_{\mathbb{H},k}(U)  \quad  \subseteq  \quad  \underset{\substack{\text{Rumin Currents}}}{ \mathcal{D}_{\mathbb{H},k}(U) }  
\end{empheq}
\caption{ } 
\label{fig:diagram_currents}
\end{figure}

\noindent
A similar figure can be obtained without requiring compact support and considering only sets with locally finite mass, meaning finite mass on compact subsets, local integrability by integration and so on. We can denote such sets with the subscription \emph{lfm} for  ``locally finite mass'' and this gives figure \ref{fig:diagram_currents_lfm}. Currents with locally finite mass have been studied, among others, by Franchi at al. (\cite{FSSC}).

\begin{figure}[!ht]
\centering
\begin{empheq}[box=\fbox]{alignat*=8}
& I_{\mathbb{H},k,\emph{lfm}}(U) &  \quad \subseteq \quad & \mathcal{R}_{\mathbb{H},k,\emph{lfm}}(U) \\
&   \quad \quad \quad  \rotatebox[origin=c]{270}{$\subseteq$} &   & \quad \quad \quad     \rotatebox[origin=c]{270}{$\subseteq$} \\
&    I_{\mathbb{H}\text{-rect},k,\emph{lfm}}(U)   & \quad \subseteq  \quad&  \mathcal{R}_{\mathbb{H}\text{-rect},k,\emph{lfm}}(U)  \\
& \quad \quad \quad    \rotatebox[origin=c]{270}{$\subseteq$}    & & \quad \quad \quad      \rotatebox[origin=c]{270}{$\subseteq$} \\
& N_{\mathbb{H},k,\emph{lfm}}(U)  &\quad \subseteq  \quad   &        R_{\mathbb{H},k,\emph{lfm}} (U)   &  \quad \subseteq \quad \mathcal{D}_{\mathbb{H},k,\emph{lfm}}(U) 
\end{empheq}
\caption{ } 
\label{fig:diagram_currents_lfm}
\end{figure}



\section{Slicing of Currents in the Heisenberg Group}\label{section:slicing}
In this section we define the notion of slices of Heisenberg currents and show, in propositions \ref{first4properties} and \ref{next3properties}, seven important properties.  
Proposition \ref{next3properties}, in particular, carries deep consequences for the possibility of developing a compactness theorem for currents in the Heisenberg group because it does not include the slices of the middle dimension $k=n$.    
Furthermore, this suggests that the study of currents on the first Heisenberg group $\mathbb{H}^1$ diverges from the other cases, because that is the only situation in which the dimension of the slice of a hypersurface, $2n-1$, coincides with the middle dimension $n$, which triggers a change in the associated differential operator in the Rumin complex.     
The most important references for the Riemannian case are sections 4.1.7 and 4.2.1 in \cite{FED} and the matching sections in \cite{MORGAN1}.

\begin{defin}
Consider an open set $U \subseteq \mathbb{H}^n$. We give the following definitions.
\begin{itemize}
\item If $f \in \mathcal{D}_{\mathbb{H}}^0(U)= C^\infty(U)$, $T \in \mathcal{D}_{\mathbb{H},k}(U)$ and $\omega \in \mathcal{D}_{\mathbb{H}}^k(U)$, then
$$
(T \mres f)(\omega):=T(f\omega).
$$
\item If $\varphi \in \mathcal{D}_{\mathbb{H}}^m(U) $, $m \leq k$, $T \in \mathcal{D}_{\mathbb{H},k}(U)$ and $ \omega \in \mathcal{D}_{\mathbb{H}}^{k-m}(U)$, then
$$
(T \mres \varphi)(\omega):=T(\varphi \wedge \omega).
$$
\item If $A \subseteq \mathbb{H}^n $ Borel set, $\chi_A :\mathbb{H}^n \to \{0,1\}  $ and $T \in R_{\mathbb{H},k}(U)$, then
$$
T \mres A(\omega)=T \mres \chi_A (\omega) :=  \int_{U }   \langle \chi_{A}  \omega  \vert \overrightarrow{ T} \rangle   d \mu_{ T}=\int_{U \cap A }   \langle  \omega  \vert \overrightarrow{ T} \rangle   d \mu_{ T}.
$$
\item If $T \in \mathcal{D}_{\mathbb{H},k}(U)$ is representable by integration,  $T=\overrightarrow{T} \wedge \mu_T$, and a function $f : U \to \mathbb{R}$ is such that $\int |f| d \mu_T < \infty$, then
$$
T \mres f := \overrightarrow{T} \wedge f\mu_T.
$$
\end{itemize}
\end{defin}

\begin{defin}\label{defin:slices}
Consider an open set $U \subseteq \mathbb{H}^n$, $f \in Lip(U,\mathbb{R})$, $t \in \mathbb{R}$ and $T\in \mathcal{D}_{\mathbb{H},k}(U)$. We define \emph{slices of $T$} the following two currents:
\begin{align*}
\langle T,f,t+ \rangle :=\left ( \partial T \right ) \mres \{ f>t \} -  \partial \left ( T \mres \{ f>t \} \right ),\\
\langle  T,f,t - \rangle :=    \partial \left ( T \mres \{ f < t \} \right ) -  \left (  \partial T  \right ) \mres \{ f <t \}.
\end{align*}
\end{defin}

\noindent
It is important to notice that, considering an open set $U \subseteq \mathbb{H}^n$, a function $f\in C^\infty (U)$ and a current $T \in R_{\mathbb{H},k}(U)$  $\big ($resp. $ \mathcal{R}_{\mathbb{H}\text{-rect},k} (U) $  or $ \mathcal{R}_{\mathbb{H},k} (U)  \big )$, we cannot imply that 
$
T \mres f \in R_{\mathbb{H},k}(U)  \left ( \text{resp. }  \mathcal{R}_{\mathbb{H}\text{-rect},k}(U)  \text{ or }  \mathcal{R}_{\mathbb{H},k}(U) \right ).
$   
The reason is that, applying a smooth function to the current, without further hypotheses, we cannot always expect the current mass 
 to remain finite. Nevertheless, something can still be said.\\

\noindent
Note that the following lemma contains three statement each (one in $R_{\mathbb{H},k}(U)$, one in $\mathcal{R}_{\mathbb{H}\text{-rect},k}(U)$ and one in  $\mathcal{R}_{\mathbb{H},k}(U)$); they are written together as the proofs are basically the same.

\begin{lem}\label{intvalued}
Consider an open set $U \subseteq \mathbb{H}^n$, $A \subseteq \mathbb{H}^n $ a Borel set   and $T \in R_{\mathbb{H},k}(U)$  $\big ($resp. $ \mathcal{R}_{\mathbb{H}\text{-rect},k} (U) $  or $ \mathcal{R}_{\mathbb{H},k} (U)  \big )$. Then
$$
T \mres \chi_A \in R_{\mathbb{H},k} (U) \left ( \text{resp. }  \mathcal{R}_{\mathbb{H}\text{-rect},k}(U)  \text{ or }  \mathcal{R}_{\mathbb{H},k}(U) \right ) .
$$
\end{lem}

\noindent
The proof of this lemma is a one-line application of the definitions.

%

\begin{lem}\label{switchsigncurrent}
Consider an open set $U \subseteq \mathbb{H}^n$, $f \in Lip(U,\mathbb{R})$, $t \in \mathbb{R}$ and $T\in \mathcal{D}_{\mathbb{H},k}(U)$. Then
\begin{align*}
\langle T,f,t+ \rangle &= \partial \left ( T \mres \{ f \leq t \} \right ) - \left ( \partial T \right ) \mres \{ f \leq t \},\\
\langle  T,f,t - \rangle &= \left (   \partial T \right ) \mres \{ f \geq t \} -  \partial \left ( T \mres \{ f \geq t \} \right ).
\end{align*}
\end{lem}

\begin{proof}
We can compute directly, using the linearity of the definition of currents,
\begin{align*}
\langle T,f,t+ \rangle &= \left ( \partial T \right ) \mres \{ f>t \} -  \partial \left ( T \mres \{ f>t \} \right )\\
&= \left ( \partial T \right ) \mres \left ( \mathbb{H}^n \setminus \{ f \leq t \}   \right )  -  \partial \left ( T \mres \left ( \mathbb{H}^n \setminus \{ f \leq t \}   \right ) \right )\\
&= \partial T - \left ( \partial T \right ) \mres \{ f \leq t \}   -  \partial \left ( T - T  \mres \{ f \leq t \}  \right )\\
&=  \partial \left ( T  \mres \{ f \leq t \}  \right )  - \left ( \partial T \right ) \mres \{ f \leq t \}  .
\end{align*}
The same can be done for $   \langle  T,f,t - \rangle .$
\end{proof}


\subsection{Properties of Slices}\label{knotn}

In the next two propositions, we show seven properties for slices of Heisenberg currents. Specifically, Proposition \ref{first4properties} holds properties similarly true in Riemannian settings (compare with 4.2.1 in \cite{FED}) and indeed we do not see an explicit use of the sub-Riemannian geometry in the proofs.  
On the other hand, Proposition \ref{next3properties}, containing the remaining 
 properties, requires 
$k \neq n$, which carries deep consequences, especially when $n=1$. Furthermore, the proof of Proposition \ref{next3properties} is way more complex than in the similar Riemannian case and requires to explicitly work with the Rumin cohomology.   %
This work follows the Riemannian theory of Federer, in particular section 4.2.1 in \cite{FED}.

\begin{prop}\label{first4properties}
Consider an open set $U \subseteq \mathbb{H}^n$,  $T\in N_{\mathbb{H},k}(U)$, $f \in Lip(U,\mathbb{R})$, and $t \in \mathbb{R}$. Then we have the following properties:
\begin{enumerate}[(1)]
  \setcounter{enumi}{-1}
\item
$\left (  \mu_T + \mu_{\partial T}   \right ) \left (  \{ f=t \}  \right ) =0 \quad \quad \text{for all } t \text{ but at most countably many}.$
\item
$\langle T,f,t+ \rangle =\langle T,f,t- \rangle  \quad \quad \text{for all } t \text{ but at most countably many}.$
\item
$\spt \langle  T,f,t+ \rangle \subseteq f^{-1} \{t\}   \cap \spt T.$
\item
$ \partial  \langle  T,f,t+ \rangle   = -  \langle \partial  T,f,t+ \rangle    .$
\end{enumerate}
\end{prop}

\begin{proof}
Property (0) holds as a general statement for measures. %
By Lemma \ref{switchsigncurrent},
\begin{align*}
\langle T,f,t+ \rangle 
&=  \partial \left ( T  \mres \{ f \leq t \}  \right )  - \left ( \partial T \right ) \mres \{ f \leq t \}  .
\end{align*}
Consider now $ T  \mres \{ f = t \}$ and notice that $ T  \mres \{ f = t \} \in R_{\mathbb{H},k} (U)$ by Lemma \ref{intvalued}, meaning that $ T  \mres \{ f = t \}$ is a current representable by integration. In particular, by property $(0)$,
$$
\left ( T  \mres \{ f = t \}  \right ) (\star)  = \int_{\{ f = t \}} \langle \star \vert \overrightarrow{T} \rangle d \mu_T =0, \quad  \text{for all } t \text{ but at most countably many}.
$$
 In the same way,  $\left ( \partial T \right ) \mres \{ f = t \} \in R_{\mathbb{H},k-1}(U) $ by hypothesis and so, again by property (0),
$$
\left (  \left ( \partial T \right )  \mres \{ f = t \}  \right ) (\star)  = \int_{\{ f = t \}} \langle \star \vert \overrightarrow{\partial T}(p) \rangle d \mu_{\partial T}(p) = 0 
$$
for all $t$  but at most countably many. So we can write that,  for all $t$  but at most countably many,
\begin{align*}
\langle T,f,t+ \rangle &=  \partial \left ( T  \mres \{ f \leq t \}  \right )  - \left ( \partial T \right ) \mres \{ f \leq t \}  \\
&=\partial \left ( T  \mres \{ f < t \}  \right )  - \left ( \partial T \right ) \mres \{ f < t \} + \partial \left ( T  \mres \{ f = t \}  \right )  - \left ( \partial T \right )  \mres \{ f = t \} \\
&=\partial \left ( T  \mres \{ f < t \}  \right )  - \left ( \partial T \right ) \mres \{ f < t \} =\langle T,f,t- \rangle .
\end{align*}
This proves property $(1)$.\\
Next we prove property $(3)$, leaving property $(2)$ as last. We have
\begin{align*}
\partial \langle T,f,t+ \rangle &= \partial \left (  \left ( \partial T \right ) \mres \{ f>t \} -  \partial \left ( T \mres \{ f>t \} \right )  \right )\\
&= \partial \left (  \left ( \partial T \right ) \mres \{ f>t \} \right ) -  \partial^2 \left ( T \mres \{ f>t \}   \right )\\
&= \partial \left (  \left ( \partial T \right ) \mres \{ f>t \} \right ) .
\end{align*}
On the other hand
\begin{align*}
- \langle  \partial T,f,t+ \rangle &= - \left (  \left [ \partial   \left ( \partial T \right ) \right ] \mres \{ f>t \}  -  \partial \left (  \left (  \partial T \right ) \mres \{ f>t \} \right )    \right )\\
&= -  \partial^2  T \mres \{ f>t \}  +  \partial \left (  \left (  \partial T \right ) \mres \{ f>t \} \right )  \\
&= \partial \left (  \left (  \partial T \right ) \mres \{ f>t \} \right ).
\end{align*}
So also property $(3)$ is verified. Only property $(2)$ is left, namely that $\spt \langle  T,f,t+ \rangle \subseteq f^{-1} \{t\}   \cap \spt T$.\\
Recalling Definition \ref{supportCurrent}, $p \in  \spt \langle  T,f,t+ \rangle$ if and only if for all neighbourhoods $U_p $ of $p$ 
 there exists a differential form $\overline\omega  \in \mathcal{D}^{k-1}_{\mathbb{H}}  (U_p) $ such that  $\langle  T,f,t+ \rangle (\overline\omega) \neq 0$ and $\spt \overline\omega \subseteq U_p$. This is the same as asking
\begin{equation}\label{recallingnot0}
 \left [  \left ( \partial T \right ) \mres \{ f>t \} -  \partial \left ( T \mres \{ f>t \} \right )  \right ] (\overline\omega) \neq 0.
\end{equation}
By contradiction, suppose that $p \notin  \spt T$, which means that there exists  a neighbourhood of $p$, $\tilde{U}_p $, 
 such that  $\tilde{U}_p \cap \spt T = \varnothing$.
By what we just noted,  $\tilde{U}_p $ is also such that  $\spt \overline\omega \subseteq \tilde{U}_p$, with $\overline\omega$ as above, and so $\spt \overline\omega   \cap \spt T = \varnothing$.

\noindent
Note then that, for $\alpha  \in \mathcal{D}^{k-1}_{\mathbb{H}}  (U) $,  $ \partial T (\alpha) = T(d_c \alpha)$ (where $d_c$ is the Rumin complex operator in general dimension, see Definition \ref{complexHn}), hence 
$\spt \partial T  \subseteq   \spt T$. Then 
$$
\spt \overline\omega   \cap \spt \partial T = \varnothing.
$$
But this is a contradiction with equation (\ref{recallingnot0}), so we have that  $p \in \spt T$.  
Consider now 
$p \in  \spt \langle  T,f,t+ \rangle$ as above and, by contradiction again, suppose than $p \notin  f^{-1} \{t\}$:
$$
p \notin  f^{-1} \{t\} \ \Longleftrightarrow \ f(p)\neq t \ \Longleftrightarrow \ f(p) >t   \text{ or }  f(p) <t.
$$
By hypothesis there exists a neighbourhood $U_p$ of $p$  
 and a differential form $\overline\omega  \in \mathcal{D}^{k-1}_{\mathbb{H}}  (U) $ such that $\spt \overline\omega \subseteq U_p$ and equation (\ref{recallingnot0}) holds. In particular, we can choose $U_p$ so that $U_p \subseteq \{f \neq t\}$. 
If $ f(p) >t$, then ${U}_p \subseteq \{f>t\}$, $\chi_{\{ f>t \}}  \overline\omega =   \overline\omega  $ and  
$$
 \left [  \left ( \partial T \right ) \mres \{ f>t \}  \right ] (\overline\omega) =
   \left ( \partial T \right )        \left ( \chi_{\{ f>t \}}  \overline\omega  \right ) =
   \left ( \partial T \right )        \left (   \overline\omega  \right ) =
T     \left (  d_c    \overline\omega \right ) .
$$
 In a similar way,
$$
 \left [   \partial \left ( T \mres \{ f>t \} \right )  \right ] (\overline\omega)=
\left ( T \mres \{ f>t \} \right )   (d_c \overline\omega) =
T   ( \chi_{\{ f>t \}} d_c \overline\omega)=
T   (  d_c \overline\omega).
$$
So
$$
\langle  T,f,t+ \rangle  (\overline\omega)= \left [  \left ( \partial T \right ) \mres \{ f>t \} -  \partial \left ( T \mres \{ f>t \} \right )  \right ] (\overline\omega) = 0
$$
which is a contradiction. If $ f(p) <t$, then $\spt \overline\omega \subseteq {U}_p \subseteq \{f < t\} \subseteq \{f \leq t\}$ and we have
$$
 \left [   \partial \left ( T \mres \{ f \leq t \} \right )  \right ] (\overline\omega)=
\left ( T \mres \{ f \leq t \} \right )   (d_c \overline\omega) =
T   ( \chi_{\{ f \leq t \}} d_c \overline\omega)=
T   (  d_c \overline\omega)
$$
and
$$
 \left [  \left ( \partial T \right ) \mres \{ f \leq t \}  \right ] (\overline\omega) =
   \left ( \partial T \right )        \left ( \chi_{\{ f \leq t \}}  \overline\omega  \right ) =
   \left ( \partial T \right )        \left (   \overline\omega  \right ) =
T     \left (  d_c    \overline\omega \right ) .
$$
Again, using Lemma \ref{switchsigncurrent},
$$
\langle  T,f,t+ \rangle  (\overline\omega)= \left [   \partial \left ( T \mres \{ f \leq t \} \right )   -   \left ( \partial T \right ) \mres \{ f \leq t \}  \right ] (\overline\omega)
= 0
$$
which is a contradiction.
 This complete the proof.
\end{proof}

\noindent
As the proof showed, the geometry of the Heisenberg group and the Rumin complex, although present, did not play a role in the previous properties. Now we show further properties for which the Rumin cohomology does play a bigger role.

\begin{prop}\label{next3properties}
Consider an open set $U \subseteq \mathbb{H}^n$, $T\in N_{\mathbb{H},k+1}(U)$, $f \in Lip(U,\mathbb{R})$, $t \in \mathbb{R}$ and $k\neq n$. Then the following properties hold:
\begin{enumerate}[(1)]
  \setcounter{enumi}{3}
\item
$M \left ( \langle T,f,t+ \rangle \right ) \leq    \ Lip(f) \liminf\limits_{h \rightarrow 0+} \frac{1}{h} \mu_T \left (   U \cap \{ t < f < t+h \}  \right )$.
\item
$\int_a^b M \left ( \langle T,f,t+ \rangle \right ) dt \leq    \ Lip(f) \mu_T \left (   U \cap \{a < f < b \}  \right ), \quad a,b \in \mathbb{R}$.
\item
$ \langle  T,f,t+ \rangle \in N_{\mathbb{H},k}(U)  \quad  \text{for a.e. } t$.
\end{enumerate}
\end{prop}

\noindent
The case $k=n$ present several differences from what we show here and, although work in that direction is ongoing, one can very easily expect differences in the final result. This comes with deep consequences as these properties are meant to be tools to help develop a compactness theorem for currents in the Heisenberg group. In detail, this corroborates that the Riemannian approach is not effective here and that new ideas are necessary. Furthermore, this also suggests that the study in the first Heisenberg group $\mathbb{H}^1$ diverges from the other cases' because, when $n=1$, then $k=n (=1)$ is the most important situation.\\

\noindent The first point is the most complicated to prove. For this reason we first contruct some machinary and show some lemmas.

\begin{lem}\label{lemmagamma}
Consider an open set $U \subseteq \mathbb{H}^n$, $f \in Lip(U,\mathbb{R})$, $t \in \mathbb{R}$, $h>0$ fixed and $s \in \mathbb{R}$. Then consider the function
\begin{equation*}
\gamma_h (s) := \frac{|s-t| - |s-(t+h)| +h}{2h}.
\end{equation*}
One can observe that
$$
\gamma_h \circ f (p) =
\begin{cases}
0 , \quad & f(p)\leq t,\\
\frac{f(p)-t}{h} , \quad & t < f(p) < t+h,\\
1 , \quad & f(p) \geq t+h,
\end{cases}
$$
$$
\gamma_h \circ f \in \ Lip(U,\mathbb{R}) \quad \text{and} \quad Lip(\gamma_h \circ f) \leq \frac{Lip(f)}{h}.
$$
\end{lem}

\begin{proof}
The computation of $\gamma_h \circ f$ follows immediately from the definition. Then, for $p,q \in U$ and considering $ t < f < t+h$, 
$$
\left | \gamma_h \circ f ( p) - \gamma_h \circ f ( q)     \right | = \left | \frac{f(p)-t}{h} - \frac{f(q)-t}{h}    \right | \leq   \frac{|f(p)-f(q)|}{h}    \leq   \frac{Lip(f)}{h} d_{\mathbb{H}}(p,q)    .
$$
This implies that $\gamma_h \circ f \in \ Lip(U,\mathbb{R})$ and, since its Lipschitz constant is the smallest for which the inequality holds, also $Lip(\gamma_h \circ f) \leq \frac{Lip(f)}{h}$ is verified.
\end{proof}

\begin{lem}\label{faseI}
Consider an open set $U \subseteq \mathbb{H}^n$,  $T\in N_{\mathbb{H},k+1}(U)$,  $f \in Lip(U,\mathbb{R})$, $t \in \mathbb{R}$, $h>0$ fixed and consider the function $\gamma_h$ defined in Lemma \ref{lemmagamma}. Then
\begin{align*}
M (\langle T,f,t+ \rangle) \leq   \liminf\limits_{h \rightarrow 0+}  M  \big ( 
 \left ( \partial T \right ) \mres{\left ( \gamma_h \circ f \right )} -  \partial \left ( T \mres{\left ( \gamma_h \circ f \right )} \right ) 
\big ).
\end{align*}
\end{lem}

\begin{proof}
Let's start by considering
\begin{align*}
 &M  \left ( \langle T,f,t+ \rangle  -    \left ( \partial T \right ) \mres{\left (\gamma_h \circ f \right )} +  \partial \left ( T \mres{\left (\gamma_h \circ f \right )} \right ) \right ) \\
= & M  \left (   \left ( \partial T \right ) \mres{\chi_{\{ f>t\}}}- \partial \left ( T \mres{\chi_{\{ f>t\}}} \right )  -    \left ( \partial T \right ) \mres{\left (\gamma_h \circ f \right )} + \partial \left ( T \mres{\left (\gamma_h \circ f \right )} \right ) \right ) \\
= & M  \left ( 
  \left ( \partial T \right )
 \mres{  \left (  \chi_{\{ f>t\}} - \gamma_h \circ f    \right )}
 + 
\partial \left (
 T \mres{ \left (      \gamma_h \circ f  -   \chi_{\{ f>t\}}    \right )} 
  \right )
\right ) \\
\leq & M  \left (   \left ( \partial T \right ) \mres{  \left (  \chi_{\{ f>t\}} - \gamma_h \circ f    \right )} \right )
 + 
M\left ( \partial \left ( T \mres{ \left (      \gamma_h \circ f  -   \chi_{\{ f>t\}}    \right )}   \right ) \right ).
\end{align*} 
Let's estimate the two terms independently. By construction $\chi_{\{ f>t\}} - \gamma_h \circ f  =0$ on $\{ f\geq t+h \}$ and $\chi_{\{ f>t\}} - \gamma_h \circ f \leq 1$ on $\{ t < f< t+h \}$, so 
$$
\chi_{\{ f>t\}} - \gamma_h \circ f \leq \chi_{\{ t < f< t+h \}}.
$$ 
Then, for $\omega \in \mathcal{D}_{\mathbb{H}}^{k-1}(U)$,
\begin{align*}
\left \vert 
\left (  \left (  \partial T \right )  \mres \left (  \chi_{\{ f>t\}} - \gamma_h \circ f   \right )  \right ) \left ( \omega   \right ) 
\right \vert 
=& 
\left \vert
 \int_{U }  \langle
\left (  \chi_{\{ f>t\}} - \gamma_h \circ f   \right ) \omega
 \vert \overrightarrow{\partial T} \rangle d \mu_{\partial T} 
\right \vert \\
\leq &  \int_{U }  
\chi_{\{ t < f< t+h \}} 
\left \vert
 \langle \omega  \vert \overrightarrow{\partial T} \rangle 
\right \vert
d \mu_{\partial T} \xrightarrow[h\to 0]{} 0 
\end{align*}
by monotone convergence theorem, which allows the limit over the integral. 
 For the second term we have:
\begin{align*}
\left \vert 
\left ( \partial \left ( T \mres{ \left (      \gamma_h \circ f  -   \chi_{\{ f>t\}}    \right )}   \right ) \right )   \left ( \omega   \right )
\right \vert 
=& 
\left \vert 
 \int_{U} \langle  \left ( \chi_{\{ f>t\}} - \gamma_h \circ f   \right ) d_Q \omega   \vert \overrightarrow{ T} \rangle d \mu_{T}
\right \vert  \\
\leq&  \int_{U}   \chi_{\{ t < f < t+h \}}  \left \vert  \langle  d_Q \omega   \vert \overrightarrow{ T} \rangle \right \vert  d \mu_{T}  \xrightarrow[h\to 0]{} 0 .
\end{align*} 
Then we have that
\begin{align*}
&
 M  \left (   \left ( \partial T \right ) \mres{  \left (  \chi_{\{ f>t\}} - \gamma_h \circ f    \right )} \right )
\\
=&
 \sup_{\norm{\omega}^* \leq 1}   \left ( \partial T \right ) \mres{  \left (  \chi_{\{ f>t\}} - \gamma_h \circ f    \right )} \left ( \omega  \right )
\\
=&
 \sup_{\norm{\omega}^* \leq 1} 
  \left ( \partial T \right ) \mres{  \left (  \chi_{\{ f>t\}} - \gamma_h \circ f    \right )} \left ( \omega  \right )
\xrightarrow[h\to 0]{} 0 .
\end{align*} 
Likewise,
\begin{align*}
M \left ( \partial \left ( T \mres{ \left (      \gamma_h \circ f  -   \chi_{\{ f>t\}}    \right )}   \right ) \right )  
  \xrightarrow[h\to 0]{} 0 .
\end{align*} 
Putting the two terms together, we get
\begin{align*}
 &
 M  \left ( \langle T,f,t+ \rangle  -    \left ( \partial T \right ) \mres{\left (\gamma_h \circ f \right )} +  \partial \left ( T \mres{\left (\gamma_h \circ f \right )} \right ) \right )
\\
\leq &
 M  \left (   \left ( \partial T \right ) \mres{  \left (  \chi_{\{ f>t\}} - \gamma_h \circ f    \right )} \right )
 + 
M\left ( \partial \left ( T \mres{ \left (      \gamma_h \circ f  -   \chi_{\{ f>t\}}    \right )}   \right ) \right ) 
 \xrightarrow[h\to 0]{} 0 .
\end{align*} 
This also means that
\begin{align*}
 M  \left ( \langle T,f,t+ \rangle  -    \left ( \partial T \right ) \mres{\left (\gamma_h \circ f \right )} +  \partial \left ( T \mres{\left (\gamma_h \circ f \right )} \right ) \right ) 
  \xrightarrow[h\to 0]{} 0 .
\end{align*} 
Finally we observe
\begin{align*}
M (\langle T,f,t+ \rangle) \leq &   M  \left ( \langle T,f,t+ \rangle  -    \left ( \partial T \right ) \mres{\left (\gamma_h \circ f \right )} +  \partial \left ( T \mres{\left (\gamma_h \circ f \right )} \right ) \right ) \\
&  +   M  \left (     \left ( \partial T \right ) \mres{\left (\gamma_h \circ f \right )} -  \partial \left ( T \mres{\left (\gamma_h \circ f \right )} \right ) \right )
\end{align*} 
and, passing to the $\liminf$ for $h \to 0$, we obtain the claim.
\end{proof}


\begin{lem}\label{lemmag_i--faseII}
Consider an open set $U \subseteq \mathbb{H}^n$, $f \in Lip(U,\mathbb{R})$, $t \in \mathbb{R}$, $h>0$ fixed and consider the function $\gamma_h$ defined in Lemma \ref{lemmagamma}. Then we can approximate $\gamma_h \circ f$ uniformly by functions $g_i\in C^\infty (U,\mathbb{R})$ (notationally $ g_i \rightrightarrows   \gamma_h \circ f  $),  so that
$$
 \spt d g_i \subseteq  \{ t < f < t+h \}  \quad \text{and} \quad  \lim\limits_{i \to \infty}  Lip(g_i) = Lip(\gamma_h \circ f).
$$
\end{lem}

\begin{proof}
By density of smooth functions, we can approximate $\gamma_h \circ f$ uniformly by smooth function $g_i\in C^\infty (U,\mathbb{R})$ and, since $\gamma_h \circ f$ is smooth and locally constant out of $ \{ t < f < t+h \}$, it follows that $g_i$ is locally constant out of $ \{ t < f < t+h \}$ as well and so that $ \spt d g_i \subseteq  \{ t < f < t+h \}$. To prove the limit, we see that, for $p,q \in U$,
\begin{align*}
\left | g_i ( p) - g_i ( q)     \right | &\leq  \left | g_i(p) - \gamma_h \circ f (p)     \right |  +   \left | \gamma_h \circ f (p) - \gamma_h \circ f (q)  \right |   + \left |  \gamma_h \circ f (q)  -g_i(q)   \right |     \\
&  \leq   Lip(\gamma_h \circ f) d_{\mathbb{H}}(p,q) +2\epsilon_i \\
&=   Lip(\gamma_h \circ f) d_{\mathbb{H}}(p,q) +2\epsilon'_i d_{\mathbb{H}}(p,q)\\
&=   \left (  Lip(\gamma_h \circ f) +2\epsilon'_i \right ) d_{\mathbb{H}}(p,q) ,
\end{align*}
with $\epsilon_i = 2\epsilon'_i d_{\mathbb{H}}(p,q)$ and $\epsilon_i  \to 0$ as $i \to \infty$ by uniform convergence.  Thus, since the Lipschitz constant of $g_i $ is the smallest for which the inequality holds,
\begin{align*}
Lip(g_i)    \leq  \ Lip(\gamma_h \circ f) +2\epsilon_i' ,
\end{align*}
and, passing to the limit,
\begin{align*}
 \lim\limits_{i \to \infty}   Lip(g_i)   \leq  \ Lip(\gamma_h \circ f)<\infty  .
\end{align*}
On the other hand
\begin{align*}
\left | \gamma_h \circ f ( p) - \gamma_h \circ f ( q)     \right | &\leq  \left | \gamma_h \circ f ( p) -g_i(p)  \right |   + \left |  g_i(p) -g_i(q)   \right |  +  \left | g_i(q) - \gamma_h \circ f ( q)     \right |\\
&  \leq  \ Lip(g_i) d_{\mathbb{H}}(p,q) +2\epsilon_i ,
\end{align*}
with $\epsilon_i  \to 0$ as $i \to \infty$. Passing to the limit,
\begin{align*}
\left | \gamma_h \circ f ( p) - \gamma_h \circ f ( q)     \right | \leq  \lim\limits_{i \to \infty}  Lip(g_i)   d_{\mathbb{H}}(p,q),
\end{align*}
so, since again the Lipschitz constant is the smallest for which the inequality holds, we get,
\begin{align*}
Lip( \gamma_h \circ f  )  \leq  \lim\limits_{i \to \infty}  Lip(g_i) .
\end{align*}
Finally, indeed $\lim\limits_{i \to \infty}  Lip(g_i) = Lip(\gamma_h \circ f).$
\end{proof}

\begin{lem}\label{faseIII}
Consider an open set $U \subseteq \mathbb{H}^n$, $T\in N_{\mathbb{H},k+1}(U)$, $f \in Lip(U,\mathbb{R})$, $t \in \mathbb{R}$, $h>0$ fixed, consider the function $\gamma_h$ defined in Lemma \ref{lemmagamma} and the functions $g_i\in C^\infty (U,\mathbb{R})$  defined in Lemma \ref{lemmag_i--faseII} so that $ g_i \rightrightarrows   \gamma_h \circ f  $. Then
\begin{align*}
M  \left ( 
 \left ( \partial T \right ) \mres{\left ( \gamma_h \circ f \right )} -  \partial \left ( T \mres{\left ( \gamma_h \circ f \right )} \right ) 
\right )
\leq  \lim_{i \to \infty}
M  \left ( 
 \left ( \partial T \right ) \mres{g_i} -  \partial \left ( T \mres{g_i} \right ) 
\right )
\end{align*}
\end{lem}

\begin{proof}
Let's first notice that
\begin{align*}
& \lim_{i \to \infty} M  \left ( 
 \left ( \partial T \right ) \mres{\left ( \gamma_h \circ f \right )} -  \partial \left ( T \mres{\left ( \gamma_h \circ f \right )} \right ) 
-
\left [ 
 \left ( \partial T \right ) \mres{g_i} -  \partial \left ( T \mres{g_i} \right ) 
\right ]
\right )  \\
=&
\lim_{i \to \infty} M  \left ( 
 \left ( \partial T \right ) \mres{\left ( \gamma_h \circ f - g_i \right )} -  \partial \left ( T \mres{\left ( \gamma_h \circ f  - g_i \right )} \right ) 
\right ) \\
 = & 0
\end{align*}
since $ g_i \rightrightarrows   \gamma_h \circ f  $. Then
\begin{align*}
&M  \left ( 
 \left ( \partial T \right ) \mres{\left ( \gamma_h \circ f \right )} -  \partial \left ( T \mres{\left ( \gamma_h \circ f \right )} \right ) 
\right ) \\
\leq &
M  \left ( 
 \big ( \partial T \right ) \mres{\left ( \gamma_h \circ f \right )} -  \partial \left ( T \mres{\left ( \gamma_h \circ f \right )} \right ) 
- \left [  \left ( \partial T \right ) \mres{g_i} -  \partial \left ( T \mres{g_i} \right ) \right ]
\big ) \\
& +
M  \left ( 
  \left ( \partial T \right ) \mres{g_i} -  \partial \left ( T \mres{g_i} \right ) 
\right ) .
\end{align*} 
Passing to the limit for $i \to \infty$, we obtain the claim.
\end{proof}


\noindent
So far we could work without explicitely using the Rumin complex operators. Now this is no more possible, as the following lemma shows.

\begin{lem}\label{L2}
Consider an open set $U \subseteq \mathbb{H}^n$, $T\in \mathcal{D}_{\mathbb{H},k+1}(U)$, $\omega \in \mathcal{D}^k_{\mathbb{H}}(U)$  and the functions $g_i \in C^\infty (U,\mathbb{R})$  defined in Lemma \ref{lemmag_i--faseII}. Also recall notations \ref{notationL} and \ref{notation:IJ}. Then
$$
\left [ \left ( \partial T \right ) \mres{g_i} -  \partial \left ( T \mres{g_i} \right ) \right ]  (\omega)  =
$$
$$
=
\begin{cases}
T   \left ( \left [ d^{(1)} g_i \wedge  \omega \right ]_{I^{k+1}}  \right ) , 
\quad [\omega ]_{I^{k}} \in \mathcal{D}^k_{\mathbb{H}}(U) = \frac{\Omega^k}{I^k},
\quad \text{if }k<n,\\
T  \Big ( 
 d^{(1)} g_i  \wedge \left ( \omega + \mathcal{L}(\omega) \wedge \theta \right )  + 
 d^{(n+1)} \left (
 \left (
 \mathcal{L}( g_i \omega) -   g_i \mathcal{L}( \omega) 
 \right )
\wedge \theta
 \right ) 
\Big ) , \\
\quad \quad  \quad  \quad \quad  \quad  \quad  \quad \quad  \quad  \quad \quad  \quad  \quad \quad  \quad  \quad [\omega ]_{I^{n}} \in \mathcal{D}^n_{\mathbb{H}}(U) = \frac{\Omega^n}{I^n},
\quad \text{if }k=n,\\
T   \left ( \left ( d^{(1)} g_i \wedge  \omega \right )_{\vert_{J^{k+1} }}  \right ) , 
\quad \omega \in \mathcal{D}^k_{\mathbb{H}}(U) = J^k,
\quad \text{if }k>n.
\end{cases}
$$
\end{lem}

\begin{proof}
Let's first note that, for $\omega \in \mathcal{D}^k_{\mathbb{H}}(U)$,
\begin{align*}
\left [ \left ( \partial T \right ) \mres{g_i} -  \partial \left ( T \mres{g_i} \right ) \right ]  (\omega)  
&=   \partial T \left ( g_i \omega \right )   -   \left ( T \mres{g_i} \right ) \left  (d_c \omega \right ) \\
&=  T \left (d_c \left ( g_i \omega \right )    -   g_i d_c \omega \right ).
\end{align*}
There are then three cases for $d_c = d_c^{(k)}$ (see Definition \ref{complexHn}), depending on $k$. First, if $k<n$, then
\begin{align*}
d_c^{(k+1)} \omega = d_Q^{(k+1)} \left [\omega \right ]_{I^k}= \left [  d^{(k+1)}\omega \right ]_{I^k}
\end{align*}
and so
\begin{align*}
d_c^{(k+1)} \left ( g_i \omega \right )    -   g_i d_c^{(k+1)} \omega 
&= \left [  d^{(k+1)} \left ( g_i \omega \right ) \right ]_{I^k} - g_i    \left [  d^{(k+1)}\omega \right ]_{I^k}\\
&= \left [  d^{(k+1)} \left ( g_i \omega \right )  - g_i    d^{(k+1)}\omega \right ]_{I^k}\\
&= \left [  d^{(1)}  g_i \wedge \omega + g_i  d^{(k+1)}  \omega   - g_i    d^{(k+1)}\omega \right ]_{I^k}\\
&= \left [  d^{(1)}  g_i \wedge \omega  \right ]_{I^k}.
\end{align*}
This proves the first case. Second, if $k>n$, similarly we have
\begin{align*}
d_c^{(k+1)} \omega = d_Q^{(k+1)} \omega= \left (  d^{(k+1)} \omega \right  ) _{\vert_{J^{k+1}}}
\end{align*}
and so
\begin{align*}
d_c^{(k+1)} \left ( g_i \omega \right )    -   g_i d_c^{(k+1)} \omega 
&= \left (  d^{(k+1)} \left ( g_i \omega \right ) \right  ) _{\vert_{J^{k+1}}} - g_i   \left (  d^{(k+1)} \omega \right  ) _{\vert_{J^{k+1}}}\\
&=\left (  d^{(k+1)} \left ( g_i \omega \right )   - g_i  d^{(k+1)} \omega \right  ) _{\vert_{J^{k+1}}}\\
&=\left (  d^{(1)}    g_i \wedge \omega  + g_i  d^{(k+1)} \omega   - g_i  d^{(k+1)} \omega \right  ) _{\vert_{J^{k+1}}}\\
&=\left (  d^{(1)}    g_i \wedge \omega \right  ) _{\vert_{J^{k+1}}}.
\end{align*}
This proves the third case. Last, we consider the case $k=n$ and we have 
\begin{align*}
d_c^{(n+1)} \omega = D \left [\omega \right ]_{I^n} &= d^{(n+1)} \left ( \omega +  L^{-1} \left (- (d \omega)_{\vert_{ {\prescript{}{}\bigwedge}^{n+1} \mathfrak{h}_1 }} \right ) \wedge \theta \right )\\
&= d^{(n+1)} \left ( \omega +  \mathcal{L} ( \omega) \wedge \theta \right ).
\end{align*}
Let's also note that
\begin{align*}
 d^{(n+1)} \left ( g_i \omega \right )  - g_i  d^{(n+1)}  \omega 
 =  d^{(1)}  g_i \wedge \omega +  g_i   d^{(n+1)} \omega     - g_i  d^{(n+1)}  \omega
=  d^{(1)}  g_i \wedge \omega,
\end{align*}
that
\begin{align*}
d^{(n+1)} \left (  \mathcal{L} ( g_i \omega) \wedge \theta \right )
= d^{(n)} \left (  \mathcal{L} ( g_i \omega)  \right ) \wedge \theta  + (-1)^{n-1}  \mathcal{L} ( g_i \omega) \wedge d^{(2)}  \theta,
\end{align*}
and that
\begin{align*}
- g_i  d^{(n+1)} \left (   \mathcal{L} ( \omega) \wedge \theta \right )
= - g_i d^{(n)} \left (  \mathcal{L} (  \omega)  \right ) \wedge \theta  - (-1)^{n-1} g_i  \mathcal{L} (  \omega) \wedge d^{(2)} \theta .
\end{align*}
Then we use all of the above and we get
\begin{align*}
d_c^{(n+1)}& \left ( g_i \omega \right )    -   g_i d_c^{(n+1)} \omega \\
=& D \left [g_i \omega \right ]_{I^n} - g_i  D \left [\omega \right ]_{I^n}\\
=& d^{(n+1)} \left ( g_i \omega +  \mathcal{L} ( g_i \omega) \wedge \theta \right ) - g_i  d^{(n+1)} \left ( \omega +  \mathcal{L} ( \omega) \wedge \theta \right )\\
=& d^{(n+1)} \left ( g_i \omega \right )  - g_i  d^{(n+1)}  \omega + d^{(n+1)} \left (  \mathcal{L} ( g_i \omega) \wedge \theta \right ) - g_i  d^{(n+1)} \left (   \mathcal{L} ( \omega) \wedge \theta \right )\\
=& d^{(1)}  g_i \wedge \omega
+ \left [   d^{(n)} \left (  \mathcal{L} ( g_i \omega)  \right )  - g_i d^{(n)} \left (  \mathcal{L} (  \omega)  \right )   \right ] \wedge \theta\\
& +  (-1)^{n-1} \left [  \mathcal{L} ( g_i \omega) -  g_i  \mathcal{L} (  \omega)   \right ] \wedge  d^{(2)}  \theta\\
=& d^{(1)}  g_i \wedge \omega
+ \left [   d^{(n)} \left (  \mathcal{L} ( g_i \omega)  \right )
  - \left (     d^{(n)} \left (   g_i  \mathcal{L} (  \omega)  \right )  -         d^{(1)}  g_i \wedge  \mathcal{L} (  \omega)        \right )
\right ] \wedge \theta\\
& +  (-1)^{n-1} \left [  \mathcal{L} ( g_i \omega) -  g_i  \mathcal{L} (  \omega)   \right ] \wedge  d^{(2)}  \theta\\
=& d^{(1)}  g_i \wedge \omega  +    d^{(1)}  g_i \wedge  \mathcal{L} (  \omega)    \wedge \theta
+   d^{(n)} \left (  \mathcal{L} ( g_i \omega)   -   g_i  \mathcal{L} (  \omega)  \right )   \wedge \theta\\
& +  (-1)^{n-1} \left (  \mathcal{L} ( g_i \omega) -  g_i  \mathcal{L} (  \omega)   \right ) \wedge  d^{(2)}  \theta\\
=& d^{(1)}  g_i \wedge \left ( \omega  +  \mathcal{L} (  \omega)   \wedge \theta \right ) 
+   d^{(n+1)}  \left ( \left (  \mathcal{L} ( g_i \omega)   -   g_i  \mathcal{L} (  \omega)  \right )   \wedge \theta \right ).
\end{align*}
This completes the proof of the lemma.
\end{proof}


\begin{lem}\label{L4}
Consider an open set $U \subseteq \mathbb{H}^n$, $T\in R_{\mathbb{H},k+1}(U)$, $\omega \in \mathcal{D}^k_{\mathbb{H}}(U)$, $k\neq n$ and the functions $g_i\in C^\infty (U,\mathbb{R})$  defined in Lemma \ref{lemmag_i--faseII}. Then
$$
\left [ \left ( \partial T \right ) \mres{g_i} -  \partial \left ( T \mres{g_i} \right ) \right ] (\omega) \leq \ Lip(g_i) \left ( T \mres \spt d g_i \right )   \left ( \sum_{j=1}^{2n} dw_j  \wedge \omega    \right ).
$$
\end{lem}

\begin{proof}
For $k<n$, by Lemma \ref{L2}, 
\begin{align*}
\left [ \left ( \partial T \right ) \mres{g_i} -  \partial \left ( T \mres{g_i} \right ) \right ] (\omega)
&= T   \left ( \left [ d g_i \wedge  \omega \right ]_{I^{k+1}}  \right ) \\
&= \int_{U \cap \spt dg_i}   \langle \left [ d g_i \wedge  \omega \right ]_{I^{k+1}}  \vert \overrightarrow{T} \rangle d \mu_T \\
&
= \int_{U \cap \spt dg_i}  \langle  d g_i \wedge  \omega  \vert \overrightarrow{T} \rangle d \mu_T .
\end{align*}
For $k>n$, by Lemma \ref{L2} again, we have a similar expression:
\begin{align*}
\left [ \left ( \partial T \right ) \mres{g_i} -  \partial \left ( T \mres{g_i} \right ) \right ] (\omega)
&=T   \left ( \left ( d g_i \wedge  \omega \right )_{\vert_{J^{k+1} }}  \right ) \\
&= \int_{U \cap \spt dg_i}   \langle \left ( d g_i \wedge  \omega \right )_{\vert_{J^{k+1} }}  \vert \overrightarrow{T} \rangle d \mu_T \\
&
= \int_{U \cap \spt dg_i}  \langle  d g_i \wedge  \omega  \vert \overrightarrow{T} \rangle d \mu_T .
\end{align*}
Recall Notation \ref{notW} and note that, as in Example \ref{df},  $d g_i = \sum_{j=1}^{2n+1} W_{j} g_i  dw_j$. If $k>n$, then $\omega \in \mathcal{D}^k_{\mathbb{H}}(U)$ is of the form $\omega =dw_{2n+1} \wedge \omega' $, $\omega' \in \Omega^{k-1}$ (see $J^k$ at Definition \ref{def_forms}). Then
$$
 d g_i \wedge  \omega  =  \sum_{j=1}^{2n+1} W_{j} g_i  dw_j  \wedge  \omega = \sum_{j=1}^{2n} W_{j} g_i  dw_j  \wedge  \omega.
$$
If $k<n$, then $T \in R_{\mathbb{H},k+1}(U)$ means that $\overrightarrow{T_p} \in  {\prescript{}{H}{\bigwedge}}_{k+1} (U)$ for $p \in U$ (see Definition \ref{repbyint}), which implies that $\overrightarrow{T} \neq W_{2n+1} \wedge V $, $V \in \Omega_{k}$. Then
$$
 \langle  d g_i \wedge  \omega  \vert \overrightarrow{T} \rangle
=
 \langle  \sum_{j=1}^{2n+1} W_{j} g_i  dw_j  \wedge  \omega  \vert \overrightarrow{T} \rangle
=
 \langle  \sum_{j=1}^{2n} W_{j} g_i  dw_j  \wedge  \omega  \vert \overrightarrow{T} \rangle.
$$
Thus, in both cases we have that
$$
 \langle  d g_i \wedge  \omega  \vert \overrightarrow{T} \rangle
=
 \langle  \sum_{j=1}^{2n} W_{j} g_i  dw_j  \wedge  \omega  \vert \overrightarrow{T} \rangle.
$$
We note that $| \nabla_{\mathbb{H}} g_i | \leq Lip( g_i )  $ and so $ W_j g_i  \leq \ Lip (g_i)$ for all $j=1,\dots,2n$.
  Indeed, using definitions \ref{dHHH} and \ref{veryfirstnabla},
\begin{align*}
W_j g_i
& \leq |W_j g_i |      = \sup_{\norm{p_0}\leq 1}  |W_j g_i (p_0) |       
 \leq  \sup_{\norm{p_0}\leq 1}  |  \nabla_\mathbb{H} g_i (p_0) |     \\ 
&   =  \sup_{\norm{p_0}\leq 1}  \left  | \left   ( d_H {g_i}_{p_0} \right )^* \right |          
   =  \sup_{\norm{p_0}\leq 1}    |   d_H {g_i}_{p_0}   |             \\
& =  \sup_{\norm{p_0}\leq 1, \ \norm{p}\leq 1}    |   d_H {g_i}_{p_0} (p)   |            \\
&= \sup_{\norm{p_0}\leq 1, \ \norm{p}\leq 1}   \lim\limits_{r \to 0+} \frac{\left |  g_i \left (p_0* \delta_r (p)  \right ) - g_i(p_0) \right | }{r}       \\
&\leq   \sup_{\norm{p_0}\leq 1, \ \norm{p}\leq 1}     Lip (g_i)   
=  Lip (g_i). \\
\end{align*}

\noindent
Then
$$
 \sum_{j=1}^{2n} W_{j} g_i  dw_j
\leq
\ Lip (g_i) \sum_{j=1}^{2n} dw_j .
$$
 Finally, for $k\neq n$,
\begin{align*}
\int_{U \cap \spt dg_i}  \langle  d g_i \wedge  \omega  \vert \overrightarrow{T} \rangle d \mu_T
&\leq \ Lip (g_i)  \int_{U \cap \spt dg_i}  \langle \sum_{j=1}^{2n} dw_j  \wedge  \omega  \vert \overrightarrow{T} \rangle d \mu_T\\
&= \ Lip(g_i) \left ( T \mres \spt d g_i \right )   \left ( \sum_{j=1}^{2n} dw_j \wedge \omega    \right ).
\end{align*}
\end{proof}

\begin{lem}\label{L5}
Consider an open set $U \subseteq \mathbb{H}^n$, $T\in N_{\mathbb{H},k+1}(U)$, $f \in Lip(U,\mathbb{R})$, $t \in \mathbb{R}$, $h>0$ and $k\neq n$.
Then
\begin{align*}
M (\langle T,f,t+ \rangle)  \leq &   \liminf\limits_{h \rightarrow 0+} \frac{Lip(f)}{h}   M  \left (
 T \mres   \left ( \chi_{ \{ t < f <t+h \} }  \sum_{j=1}^{2n} dw_j  \right )
\right ).
\end{align*}
\end{lem}

\begin{proof}
By Lemmas \ref{faseI} and \ref{faseIII} we have that
\begin{align*}
M (\langle T,f,t+ \rangle) \leq & \liminf\limits_{h \rightarrow 0+}
M  \left ( 
 \left ( \partial T \right ) \mres{\left ( \gamma_h \circ f \right )} -  \partial \left ( T \mres{\left ( \gamma_h \circ f \right )} \right ) 
\right )\\
\leq &  \liminf\limits_{h \rightarrow 0+}  \lim_{i \to \infty}
M  \left ( 
 \left ( \partial T \right ) \mres{g_i} -  \partial \left ( T \mres{g_i} \right ) 
\right ).
\end{align*}
 Then, by Lemma \ref{L4},
$$
M \left ( \left ( \partial T \right ) \mres{g_i} -  \partial \left ( T \mres{g_i} \right ) \right ) \leq \ Lip(g_i) 
M  \left (
 T \mres   \left ( \chi_{ \spt d g_i }  \sum_{j=1}^{2n} dw_j  \right )
\right ).
$$
Notice that
\begin{align*}
&M  \left (
 T \mres   \left ( \chi_{ \spt d g_i }  \sum_{j=1}^{2n} dw_j  \right )
\right )       = \sup_{\substack{ \norm{\omega}^*\leq 1, \\ \omega \in \mathcal{D}_{\mathbb{H}}^k(U)  }}
\int_{U} \langle   \sum_{j=1}^{2n} dw_j   \wedge \left ( \chi_{ \spt d g_i } \omega   \right )    \vert  \overrightarrow{T}  \rangle d\mu_T
\end{align*}
and denote $\omega'=\chi_{ \spt d g_i } \omega$. Then $\norm{\omega'}^*\leq 1$ and $\omega' \in \mathcal{D}_{\mathbb{H}}^k(U \cap  \spt d g_i ) $. Thus, since $ \spt d g_i \subseteq  \{ t < f < t+h \} $  by Lemma \ref{lemmag_i--faseII},
\begin{align*}
M  \left (
 T \mres   \left ( \chi_{ \spt d g_i }  \sum_{j=1}^{2n} dw_j  \right )
\right )  &=\sup_{\substack{ \norm{\omega'}^*\leq 1, \\ \omega' \in \mathcal{D}_{\mathbb{H}}^k(U)  }}
\int_{U \cap \spt dg_i} \langle   \sum_{j=1}^{2n} dw_j   \wedge \omega'    \vert  \overrightarrow{T}  \rangle d\mu_T\\
&\leq \sup_{\substack{ \norm{\omega'}^*\leq 1, \\ \omega' \in \mathcal{D}_{\mathbb{H}}^k(U)  }}
\int_{U \cap \{ t < f < t+h \} } \langle   \sum_{j=1}^{2n} dw_j   \wedge \omega'    \vert  \overrightarrow{T}  \rangle d\mu_T\\
&= \sup_{\substack{ \norm{\omega'}^*\leq 1, \\ \omega' \in \mathcal{D}_{\mathbb{H}}^k(U)  }}
\int_{U} \langle   \sum_{j=1}^{2n} dw_j   \wedge \chi_{ \{ t < f < t+h \} } \omega'    \vert  \overrightarrow{T}  \rangle d\mu_T\\
&= M  \left (
 T \mres   \left ( \chi_{ \{ t < f < t+h \} }  \sum_{j=1}^{2n} dw_j  \right )
\right ).
\end{align*}
Putting the pieces together, we get
\begin{align*}
M (\langle T,f,t+ \rangle)  \leq &    \liminf\limits_{h \rightarrow 0+} \lim_{i \to \infty}
 Lip(g_i)  M  \left (
 T \mres   \left ( \chi_{ \{ t < f < t+h \} }  \sum_{j=1}^{2n} dw_j  \right )
\right ).
\end{align*}
By Lemmas \ref{lemmagamma} and \ref{lemmag_i--faseII} again, $  \lim\limits_{i \to \infty}  Lip(g_i) = Lip(\gamma_h \circ f)   \leq \frac{Lip(f)}{h}  $, which gives
\begin{align*}
M (\langle T,f,t+ \rangle)  \leq &  \liminf\limits_{h \rightarrow 0+} \frac{Lip(f)}{h}   M  \left (
 T \mres   \left ( \chi_{ \{ t < f < t+h \} }  \sum_{j=1}^{2n} dw_j  \right )
\right ).
\end{align*}
\end{proof}


\noindent Finally we have all the instruments to prove the proposition.

\begin{proof}[Proof of Proposition \ref{next3properties}]
Consider the function $\gamma_h$ defined in Lemma \ref{lemmagamma} and the functions $g_i\in C^\infty (U,\mathbb{R})$  defined in Lemma \ref{lemmag_i--faseII} so that $ g_i \rightrightarrows   \gamma_h \circ f  $.  Then, by Lemma \ref{L5},
\begin{align*}
M (\langle T,f,t+ \rangle) & \leq   \liminf\limits_{h \rightarrow 0+} \frac{Lip(f)}{h}   M  \left (
 T \mres   \left ( \chi_{ \{ t < f <t+h \} }  \sum_{j=1}^{2n} dw_j  \right )
\right )\\
&=    \liminf\limits_{h \rightarrow 0+}  \frac{Lip(f)}{h} 
 \sup_{\substack{ \norm{\omega}^*\leq 1, \\ \omega \in \mathcal{D}_{\mathbb{H}}^k(U)  }}
    \left [  T \mres \left (  \chi_{\{ t < f < t+h \}}  \sum_{j=1}^{2n} dw_j \right )  \right ] (\omega)     \\
&=\liminf\limits_{h \rightarrow 0+}   \frac{Lip(f)}{h}  
 \sup_{\substack{ \norm{\omega}^*\leq 1, \\ \omega \in \mathcal{D}_{\mathbb{H}}^k(U)  }}
\left [  T \mres \left (  \chi_{\{ t < f < t+h \}}  \right )  \right ] \left (  \sum_{j=1}^{2n} dw_j \wedge \omega \right ) .
\end{align*}
Denote $\omega' =  \sum_{j=1}^{2n} dw_j \wedge \omega$, $\omega' \in  \mathcal{D}_{\mathbb{H}}^{k+1}(U) $ and consider
\begin{align*}
\norm{\omega'}^*=\norm{ \sum_{j=1}^{2n} dw_j \wedge \omega }^* 
=\sup_{\substack{ v \in   {\prescript{}{H}\bigwedge}_{k+1} (U), \\ \vert v \vert \leq 1  }}
     \langle    \sum_{j=1}^{2n} dw_j \wedge  \omega   \vert v \rangle  .
\end{align*}
By Definitions  \ref{HLambda} and \ref{masscomass}, $v$ is a simple $(k+1)$-vector, so we can write $v=\rho W_{i_1} \wedge \dots \wedge W_{i_{k+1}}$ with $\vert \rho \vert \leq 1$ and $1 \leq i_1 < \dots < i_{k+1} \leq 2n+1 $. We see that
\begin{align*}
\left \vert  \langle    \sum_{j=1}^{2n} dw_j \wedge  \omega  \vert v \rangle  \right \vert 
&\leq  \vert  \rho  \vert
\left \vert 
 \sum_{\sigma \in Sh(1,k+1)} \sgn(\sigma)
 \langle    \sum_{j=1}^{2n} dw_j   \vert W_{i_{\sigma(1)}} \rangle
\langle   \omega \vert    W_{i_{\sigma(2)}} \wedge \dots \wedge W_{i_{\sigma(k+1)}}     \rangle
\right \vert 
\end{align*}
where $Sh(1,k+1)$ is the set of $(1,k+1)$-shuffles. 
 Notice that
\begin{align*}
\left \vert 
 \langle    \sum_{j=1}^{2n} dw_j   \vert W_{i_{\sigma(1)}} \rangle
\right \vert 
\leq
\left \vert 
 \langle    dw_{i_{\sigma(1)}}  \vert W_{i_{\sigma(1)}} \rangle
\right \vert 
\leq 1
\end{align*}
and, since $\norm{\omega}^* \leq 1$,
\begin{align*}
\left \vert 
\langle   \omega \vert    W_{i_{\sigma(2)}} \wedge \dots \wedge W_{i_{\sigma(k+1)}}     \rangle
\right \vert 
\leq
\norm{\omega}^*   \left \vert   W_{i_{\sigma(2)}} \wedge \dots \wedge W_{i_{\sigma(k+1)}} \right \vert 
\leq 1.
\end{align*}
This means that
\begin{align*}
\left \vert  \langle    \sum_{j=1}^{2n} dw_j \wedge  \omega  \vert v \rangle  \right \vert 
&\leq \left \vert 
 \sum_{\sigma \in Sh(1,k+1)} \sgn(\sigma)
\right \vert 
\leq 1.
\end{align*}
because $\sgn(\sigma)$ changes at every step of the sum so the absolute value of the sum can be, at the end, only $0$ or $1$.
Finally we get
\begin{align*}
\norm{\omega'}^* \leq 1.
\end{align*}
Then
\begin{align*}
M (\langle T,f,t+ \rangle) & \leq   \liminf\limits_{h \rightarrow 0+}  \frac{Lip(f)}{h}  
 \sup_{\substack{ \norm{\omega}^*\leq 1, \\ \omega \in \mathcal{D}_{\mathbb{H}}^k(U)  }}
  \left [  T \mres \left (  \chi_{\{ t < f < t+h \}}  \right )  \right ] \left (  \sum_{j=1}^{2n} dw_j \wedge \omega \right )\\
&\leq  \liminf\limits_{h \rightarrow 0+}   \frac{Lip(f)}{h}
 \sup_{\substack{ \norm{\omega'}^*\leq 1, \\ \omega' \in \mathcal{D}_{\mathbb{H}}^{k+1}(U)  }}
   \left [  T \mres \left (  \chi_{\{ t < f < t+h \}}  \right )  \right ] (  \omega') \\
&\leq   \liminf\limits_{h \rightarrow 0+}     \frac{Lip(f)}{h}  M \left (  T \mres \{ t < f < t+h \}  \right ) \\
&=     \liminf\limits_{h \rightarrow 0+}     \frac{Lip(f)}{h}  \mu_T \left (   U \cap \{ t < f < t+h \}  \right ).
\end{align*}
by Proposition \ref{masstoint}. Then
\begin{align*}
M \left (  \langle T,f,t+ \rangle  \right )   \leq    \ {Lip}(f) \liminf\limits_{h \rightarrow 0+} \frac{1}{h} \mu_T \left (   U \cap \{ t < f < t+h \}  \right ).
\end{align*}
This proves property (4). The other two properties follow quickly. To prove property (5) we proceed as in 4.11 in \cite{MORGAN1}. Consider $F(t)=\mu_T \left (   U \cap \{  f < t \} \right )$, an increasing monotone function with derivative almost everywhere. 
\begin{align*}
&  \ {Lip}(f)  \mu_T \left (   U \cap \{ a < f < b \} \right ) \\
&=  \ {Lip}(f)  \left ( \mu_T \left (   U \cap \{  f < b \} \right ) -  \mu_T \left (   U \cap \{  f \leq a \} \right )     \right )\\
&=   \ {Lip}(f)  \left ( F( b) - \lim_{s\to a+} F( a )     \right )\\
&\geq   \ {Lip}(f) \int_a^b   F'(t) dt \\
&\geq \int_a^b M \left ( \langle T,f,t+ \rangle \right ) dt .
\end{align*}
where we apply property (4) to the last inequality; this proves property (5). Another way to finish is to observe that $t \to M \left ( \langle T,f,t+ \rangle \right )$ is measurable, which holds true because of the definitions of mass and slice, and conclude similarly. By Proposition \ref{masstoint} and since $T \in N_{\mathbb{H},k+1}(U)$, we have that $\mu_T \left (   U \cap \{ a < f < b \} \right ) < \infty$. Then, by property (5),
$$
M \left ( \langle T,f,t+ \rangle \right ) <\infty \quad  \text{for a.e. } t.
$$
Finally, by property (3) in Proposition \ref{first4properties} and by repeating the previous argument for $\partial T$,
$$
M \left ( \partial \langle T,f,t+ \rangle \right )  = M \left ( - \langle \partial T,f,t+ \rangle \right ) <\infty \quad  \text{for a.e. } t
$$
Thus property (6) holds.
\end{proof}


\subsection{The case $k=n$}

Note that lemmas \ref{L4} and \ref{L5} and Proposition \ref{next3properties} did not deal with the case $k=n$, the more challenging one. If $k=n$ we can consider $U \subseteq \mathbb{H}^n$ an open set, $T\in \mathcal{D}_{\mathbb{H},n+1}(U)$,  
 $[\omega ]_{I^{n}} \in \mathcal{D}^n_{\mathbb{H}}(U) = \frac{\Omega^n}{I^n}$,  and the functions $g_i \in C^\infty (U,\mathbb{R})$  defined in Lemma \ref{lemmag_i--faseII}. By Lemma \ref{L2}, then
\begin{align}\label{newbeginning}
\begin{aligned}
&\left [ \left ( \partial T \right ) \mres{g_i} -  \partial \left ( T \mres{g_i} \right ) \right ] \left  (  [\omega ]_{I^{n}}  \right )  \\
=&
T  \Big (    d^{(1)} g_i  \wedge \left ( \omega + \mathcal{L}(\omega) \wedge \theta \right )  +    d^{(n+1)} \left (   \left (   \mathcal{L}( g_i \omega) -   g_i \mathcal{L}( \omega)    \right )  \wedge \theta   \right )  \Big ) .
\end{aligned}
\end{align}

\noindent
The right hand side can be partially rewritten using the following lemma.

\begin{lem}\label{L3}
Let  $U \subseteq \mathbb{H}^n$ open,  $\omega \in \Omega^n$   and the functions $g_i \in C^\infty (U,\mathbb{R})$  defined in Lemma \ref{lemmag_i--faseII}. Also recall Notation \ref{notationL}.  Then
$$
\mathcal{L}( g_i \omega) -   g_i \mathcal{L}( \omega) 
= L^{-1}   \left ( - \left ( d^{(1)} g_i \wedge  \omega \right )_{\vert_{ {\prescript{}{}\bigwedge}^{n+1} \mathfrak{h}_1 }}  \right ) .
$$
\end{lem}
\begin{proof}
By the definition of  $\mathcal{L}$ and using twice the linearity of $L^{-1}$,
\begin{align*}
 \mathcal{L}( g_i \omega) &= L^{-1}   \left ( - \left ( d^{(n+1)} (g_i \omega ) \right )_{\vert_{ {\prescript{}{}\bigwedge}^{n+1} \mathfrak{h}_1 }}  \right ) \\
&= L^{-1}   \left ( - \left ( d^{(1)} g_i \wedge \omega + g_i d^{(n+1)} \omega  \right )_{\vert_{ {\prescript{}{}\bigwedge}^{n+1} \mathfrak{h}_1 }}  \right ) \\
&= L^{-1}   \left ( - \left ( d^{(1)} g_i \wedge \omega \right )_{\vert_{ {\prescript{}{}\bigwedge}^{n+1} \mathfrak{h}_1 }}  - \left ( g_i d^{(n+1)} \omega  \right )_{\vert_{ {\prescript{}{}\bigwedge}^{n+1} \mathfrak{h}_1 }}  \right ) \\
&= L^{-1}   \left ( - \left ( d^{(1)} g_i \wedge \omega \right )_{\vert_{ {\prescript{}{}\bigwedge}^{n+1} \mathfrak{h}_1 }} \right )    +    
g_i L^{-1}   \left (  - \left (  d^{(n+1)} \omega  \right )_{\vert_{ {\prescript{}{}\bigwedge}^{n+1} \mathfrak{h}_1 }}  \right )\\
&=L^{-1}   \left ( - \left ( d^{(1)} g_i \wedge \omega \right )_{\vert_{ {\prescript{}{}\bigwedge}^{n+1} \mathfrak{h}_1 }} \right )    +    
g_i  \mathcal{L}(  \omega) . 
\end{align*}
\end{proof}

\noindent
Furthermore, one can observe that the right hand side of equation (\ref{newbeginning}) is not null because of the following lemma. This observation was not needed in the cases of $k\neq n$ as the definition of $\mathcal{D}_{\mathbb{H}^*}(U)$ didn't change between $k$ and $k+1$, making the step immediate.

\begin{lem}\label{appartenenza}
Consider an open set $U \subseteq \mathbb{H}^n$, $T\in \mathcal{D}_{\mathbb{H},n+1}(U)$, $\omega \in \Omega^n$  
 and the functions $g_i \in C^\infty (U,\mathbb{R})$  defined in Lemma \ref{lemmag_i--faseII}. Also recall Notation \ref{notationL}. Then
\begin{align}\label{condJ}
    d^{(1)} g_i  \wedge \left ( \omega + \mathcal{L}(\omega) \wedge \theta \right )  +    d^{(n+1)} \left (   \left (   \mathcal{L}( g_i \omega) -   g_i \mathcal{L}( \omega)    \right )  \wedge \theta   \right )  \in J^{n+1}.
\end{align}
\end{lem}

\begin{proof}
Consider 
$$
\omega = \sum_{1\leq l_1 \leq \dots \leq l_n \leq 2n+1 } \omega_{ l_1  \dots  l_n} dw_{l_1} \wedge \dots \wedge dw_{l_n}  
\quad \text{and} \quad
d g_i = \sum_{j=1}^{ 2n+1 } W_j g_i dw_j.  
$$
Assume first that 
$\omega$ has some components without $\theta$. We compute
\begin{align*}
 d g_i \wedge \omega
=& \sum_{j=1}^{ 2n+1 }  \sum_{1\leq l_1 \leq \dots \leq l_n \leq 2n+1 } 
 W_j g_i  \omega_{ l_1  \dots  l_n} 
 dw_j  \wedge  dw_{l_1} \wedge \dots  \wedge dw_{l_n} 
\end{align*}
and
\begin{align*}
 -\left ( d g_i \wedge \omega \right )_{\vert_{{\prescript{}{}\bigwedge}^{n+1} \mathfrak{h}_1 }}
=& - \sum_{j=1}^{ 2n }  \sum_{1\leq l_1 \leq \dots \leq l_n \leq 2n } 
 W_j g_i  \omega_{ l_1  \dots  l_n} 
 dw_j  \wedge  dw_{l_1} \wedge \dots  \wedge dw_{l_n} .
\end{align*}
Notice that the $dw_{l_m}$'s are $n$ different basis elements of $\Omega^1$ and they always have their counterpart $dw_{l_m+n}$ among the $ dw_j $'s, since $j=1,\dots,2n$. Hence we can write 
\begin{align*}
 -\left ( d g_i \wedge \omega \right )_{\vert_{{\prescript{}{}\bigwedge}^{n+1} \mathfrak{h}_1 }}
= - \sum_{j=1}^{ n }  dw_j  \wedge  dw_{j+n} \wedge \gamma
= d\theta \wedge \gamma
=\gamma \wedge d\theta,
\end{align*}
where $\gamma \in  {\prescript{}{}\bigwedge}^{n-1} \mathfrak{h}_1$. It follows that
\begin{align*}
 \mathcal{L}( g_i \omega) -   g_i \mathcal{L}( \omega)
=& L^{-1} \left ( -\left ( d g_i \wedge \omega \right )_{\vert_{{\prescript{}{}\bigwedge}^{n+1} \mathfrak{h}_1 }} \right )  =\gamma.
\end{align*}
Next,
\begin{align}\label{thetawedge}
\begin{aligned}
& \theta \wedge \big [
 d g_i  \wedge \left ( \omega + \mathcal{L}(\omega) \wedge \theta \right )  +    d \left (   \left (   \mathcal{L}( g_i \omega) -   g_i \mathcal{L}( \omega)    \right )  \wedge \theta   \right ) 
 \big ] \\
=&
\theta \wedge   d g_i  \wedge \left ( \omega + \mathcal{L}(\omega) \wedge \theta \right ) 
 + 
\theta \wedge   d \left (   \left (   \mathcal{L}( g_i \omega) -   g_i \mathcal{L}( \omega)    \right )  \wedge \theta   \right ) \\
=&
\theta \wedge   d g_i  \wedge \omega 
 + 
\theta \wedge      \left (   \mathcal{L}( g_i \omega) -   g_i \mathcal{L}( \omega)    \right )  \wedge d\theta   \\
=&
\theta \wedge  \left ( d g_i \wedge \omega \right )_{\vert_{{\prescript{}{}\bigwedge}^{n+1} \mathfrak{h}_1 }}
 + 
\theta \wedge      \gamma \wedge d\theta   \\
=&
\theta \wedge  \left ( d g_i \wedge \omega \right )_{\vert_{{\prescript{}{}\bigwedge}^{n+1} \mathfrak{h}_1 }}
 -
\theta \wedge     \left ( d g_i \wedge \omega \right )_{\vert_{{\prescript{}{}\bigwedge}^{n+1} \mathfrak{h}_1 }}\\
=&0.
\end{aligned}
\end{align}
This proves the first condition for belonging to $J^{n+1}$. For the second condition, we apply the operator $d$ to (\ref{thetawedge}) and get
\begin{align*}
0=
& d \big [ \theta \wedge \left [
 d g_i  \wedge \left ( \omega + \mathcal{L}(\omega) \wedge \theta \right )  +    d \left (   \left (   \mathcal{L}( g_i \omega) -   g_i \mathcal{L}( \omega)    \right )  \wedge \theta   \right ) 
 \right ]  \big ]\\
=& d  \theta \wedge  \left [
 d g_i  \wedge \left ( \omega + \mathcal{L}(\omega) \wedge \theta \right )  +   
 d \left (   \left (   \mathcal{L}( g_i \omega) -   g_i \mathcal{L}( \omega)    \right )  \wedge \theta   \right ) 
 \right ] \\ 
&+ \theta \wedge d \left [
 d g_i  \wedge \left ( \omega + \mathcal{L}(\omega) \wedge \theta \right )  +    d \left (   \left (   \mathcal{L}( g_i \omega) -   g_i \mathcal{L}( \omega)    \right )  \wedge \theta   \right ) 
 \right ] \\
=&  d  \theta \wedge \left [
 d g_i  \wedge \left ( \omega + \mathcal{L}(\omega) \wedge \theta \right )  +    d \left (   \left (   \mathcal{L}( g_i \omega) -   g_i \mathcal{L}( \omega)    \right )  \wedge \theta   \right ) 
 \right ] \\
&+  \theta \wedge   d g_i  \wedge d \left ( \omega + \mathcal{L}(\omega) \wedge \theta \right )  .
\end{align*}
We know by Observation 2.1.11 in \cite{GClicentiate} 
 that $d \left ( \omega + \mathcal{L}(\omega) \wedge \theta \right ) \in J^{n+1}$, meaning that
$$
 \theta \wedge d \left ( \omega + \mathcal{L}(\omega) \wedge \theta \right )  =0 .
$$
Then we verify the second condition by concluding that
$$
d  \theta \wedge  \big [
 d g_i  \wedge \left ( \omega + \mathcal{L}(\omega) \wedge \theta \right )  +    d \left (   \left (   \mathcal{L}( g_i \omega) -   g_i \mathcal{L}( \omega)    \right )  \wedge \theta   \right ) 
 \big ] =0.
$$
Assume now that   $\omega$ has no components without $\theta$. Then $\omega = \theta \wedge \beta$, $\beta \in  \Omega^{n-1} $ and, trivially by definition of $ \mathcal{L}$, $ \mathcal{L}(\omega)=0= \mathcal{L}(g_i \omega)$. It follows that condition \ref{condJ} is reduced to prove that $ d g_i  \wedge  \omega 
 \in J^{n+1}$. We see immediately that
\begin{align*}
\theta  \wedge  d g_i  \wedge  \omega 
=0
\end{align*}
and, applying $d$, that
\begin{align*}
0=& d \big ( \theta \wedge  d g_i  \wedge  \omega   \big ) \\
=& d  \theta \wedge d g_i  \wedge  \omega   + \theta \wedge d g_i  \wedge d  \omega  .
\end{align*}
We know, again by Observation 2.1.11 in \cite{GClicentiate}, 
 that $ d\omega=d \left ( \omega + \mathcal{L}(\omega) \wedge \theta \right )  \in J^{n+1} $ and so
\begin{align*}
 \theta \wedge d g_i  \wedge d  \omega    =0 .
\end{align*}
Then
\begin{align*}
d  \theta \wedge  d g_i  \wedge  \omega    =0 .
\end{align*}
This completes the proof.
\end{proof}

\noindent
From the discussion in section \ref{knotn}, it is clear that the difficulty of the case with $k=n$ lies in having 
 an inequality of the kind of property (4) in Proposition \ref{next3properties}. This means that, for $T\in N_{\mathbb{H},n+1}(U)$, $f \in Lip(U,\mathbb{R})$ and $t \in \mathbb{R}$, we wish to estimate $M \left ( \langle T,f,t+ \rangle \right )$ from above with a quantity including $ \liminf\limits_{h \rightarrow 0+} \frac{1}{h} \mu_T \left (   U \cap \{ t < f < t+h \}  \right )$.
As a start, by Lemma \ref{faseIII}, we already know that
\begin{align*}
M (\langle T,f,t+ \rangle) \leq &  \liminf\limits_{h \rightarrow 0+}  \lim_{i \to \infty}
M  \left ( 
 \left ( \partial T \right ) \mres{g_i} -  \partial \left ( T \mres{g_i} \right ) 
\right ).
\end{align*}
where, by Lemmas \ref{L2} and \ref{L3}, for $\omega \in \Omega^n$,
\begin{align}
\begin{aligned}\label{2piecesT}
\left [ \left ( \partial T \right ) \mres{g_i} -  \partial \left ( T \mres{g_i} \right ) \right ]  ([\omega ]_{I^{n}})  
=&
T  \Bigg (  d^{(1)} g_i  \wedge \left ( \omega + \mathcal{L}(\omega) \wedge \theta \right ) \\
&+   d^{(n+1)} \left (  
L^{-1} \left ( - \left ( d^{(1)} g_i \wedge  \omega \right )_{\vert_{ {\prescript{}{}\bigwedge}^{n+1} \mathfrak{h}_1 }}  \right )
  \wedge \theta \right ) \Bigg ) .
\end{aligned}
\end{align}
To proceed from here, we would need some results to replace lemmas \ref{L4} and \ref{L5}, which are currently missing. The case $k=n$ is the subject of ongoing research work.







\bibliography{Bibliography_United_Thesis_190520} 
\bibliographystyle{abbrv}

\end{document}